\documentclass[1p,twoside]{elsarticle}

\usepackage{amsmath,amssymb}
\usepackage{amsthm}
\usepackage{latexsym}
\usepackage{graphicx}
\usepackage[usenames,dvipsnames,svgnames,table]{xcolor}
\usepackage{subfigure}
\usepackage{hyperref}
\usepackage{multirow}
\usepackage{algorithm}
\usepackage{algorithmic}
\usepackage{url}
\usepackage{soul}
\usepackage{epstopdf}

\newcommand{\vf}{\boldsymbol{f}}
\newcommand{\vg}{\boldsymbol{g}}

\newcommand{\vF}{\boldsymbol{F}}
\newcommand{\vG}{\boldsymbol{G}}

\newcommand{\valpha}{\boldsymbol{\alpha}}
\newcommand{\vbeta}{\boldsymbol{\beta}}

\newcommand{\NN}{\mathbb{N}}
\newcommand{\ZZ}{\mathbb{Z}}

\newcommand{\RR}{\mathbb{R}}
\newcommand{\CC}{\mathbb{C}}

\newcommand{\ud}{\textup{d}}

\newcommand{\argmin}{\operatornamewithlimits{argmin}}

\newtheorem{theorem}{Theorem}[section]
\newtheorem{defn}[theorem]{Definition}

\newtheorem{thm}{Theorem}[section]

\newtheorem{lemma}{Lemma}

\renewcommand*{\theclaim}{0}

\theoremstyle{remark}
\newtheorem*{rem}{Remark}

\begin{document}

\begin{frontmatter}

\title{Convex Optimization approach to signals with fast varying instantaneous frequency}

\author[AM,MK]{Matthieu~Kowalski}
\ead{matthieu.kowalski@lss.supelec.fr}

\author[AM]{Adrien~Meynard}

\author[HTW]{Hau-tieng~Wu}
\ead{hauwu@math.toronto.edu}

\address[AM]{Laboratoire des Signaux et Syst\`emes -- Univ Paris-Sud -- CNRS -- CentraleSupelec}
\address[MK]{Parietal project-team, INRIA, Neurospin, CEA-Saclay, France}
\address[HTW]{Department of Mathematics, University of Toronto, Toronto, Ontario, Canada}

\begin{abstract}
Motivated by the limitation of analyzing oscillatory signals composed of multiple components with fast-varying instantaneous frequency, we approach the time-frequency analysis problem by optimization. Based on the proposed adaptive harmonic model, the time-frequency representation of a signal is obtained by directly minimizing a functional, which involves few properties an ``ideal time-frequency representation'' should satisfy, for example, the signal reconstruction and concentrative time frequency representation. FISTA (Fast Iterative Shrinkage-Thresholding Algorithm) is applied to achieve an efficient numerical approximation of the functional. We coin the algorithm as {\it Time-frequency bY COnvex OptimizatioN} (Tycoon). The numerical results confirm the potential of the Tycoon algorithm.
\end{abstract}

\end{frontmatter}

\section{Introduction}

Extracting proper features from the collected dataset is the first step toward data analysis. Take an oscillatory signal as an example. We might ask how many oscillatory components inside the signal, how fast each component oscillates, how strong each component is, etc. Traditionally, Fourier transform is commonly applied to answer this question. However, it has been well known for a long time that when the signal is not composed of harmonic functions, then Fourier transform might not perform correctly. Specifically, when the signal satisfies $f(t)=\sum_{k=1}^KA_k(t)\cos(2\pi\phi_k(t))$, where $K\in\NN$, $A_k(t)>0$ and $\phi'_k(t)>0$ but $A_k(t)$ and $\phi'_k(t)$ are not constants, the momentary behavior of the oscillation cannot be captured by the Fourier transform. 
A lot of efforts have been made in the past few decades to handle this problem. Time-frequency (TF) analysis based on different principals \cite{Flandrin:1999} has attracted a lot of attention in the field and many variations are available. Well known examples include short time Fourier transform (STFT), continuous wavelet transform (CWT), Wigner-Ville distribution (WVD), chirplet transform \cite{Mann_Haykin:1995}, S-transform \cite{Stockwell_Mansinha_Lowe:1996}, etc. 

While these methods are widely applied in many fields, they are well known to be limited, again, by the Heisenberg uncertainty principle or the mode mixing problem caused by the interference known as the Moire patterns \cite{Flandrin:1999}. 
To alleviate the shortage of these analyses, in the past decades several solutions were proposed. For example, the empirical mode decomposition (EMD) \cite{Huang_Shen_Long_Wu_Shih_Zheng_Yen_Tung_Liu:1998} was proposed to study the dynamics hidden inside an oscillatory signal; however, its mathematical foundation is still lacking at this moment and several numerical issues cannot be ignored. Variations of EMD, like \cite{Wu_Huang:2009,Oberlin_Meignen_Perrier:2012,Gilles:2013,Pustelnik_Borgnat_Flandrin:2014,Dragomiretskiy_Zosso:2014}, were proposed to improve EMD.
The sparsity approach \cite{Hou_Shi:2011,Hou_Shi:2013a,Hou_Shi:2013b,Tavallali_Hou_Shi:2014} and iterative convolution-filtering \cite{Lin_Wang_Zhou:2009,Huang_Wang_Yang:2009,Cicone_Liu_Zhou:2014,Cicone_Zhou:2015} are another algorithms proposed to capture the flavor of the EMD, which have solid mathematical supports. The problem could also be discussed via other approaches, like the optimized window approach \cite{Ricaud_Stempfel_Torresani:2014}, nonstationary Gabor frame \cite{Balazs_Dorfler_Jaillet_Holighaus_Velasco:2011}, ridge approach \cite{Ricaud_Stempfel_Torresani:2014},  the approximation theory approach \cite{Chui_Mhaskar:2015}, non-local mean approach \cite{Galiano_Velasco:2014} and time-varying autoregression and moving average approach \cite{DeLivera_Alysha_Hyndman_Snyder:2011}, to name but a few. 
Among these approaches, the reassignment technique \cite{Kodera_Gendrin_Villedary:1978,Auger_Flandrin:1995,Chassande-Mottin_Auger_Flandrin:2003,Auger_Chassande-Mottin_Flandrin:2012} and the synchrosqueezing transform (SST) \cite{Daubechies_Maes:1996,Daubechies_Lu_Wu:2011,Chen_Cheng_Wu:2014} have attracted more and more attention in the past few years. The main motivation of the reassignment technique is to improve the resolution issue introduced by the Heisenberg principal -- 
the STFT coefficients are reallocated in {\it both} frequency axis and time axis according to their local phase information, which leads to the reassignment technique. The same reassignment idea can be applied to a very general settings like Cohen's class, affine class, etc \cite{Flandrin:2001}. SST is a special reassignment technique; in SST, the STFT or CWT coefficients are reassigned {\it only} on the frequency axis \cite{Daubechies_Maes:1996,Daubechies_Lu_Wu:2011,Chen_Cheng_Wu:2014} so that the causality is preserved and hence a real time algorithm is possible \cite{Chui_Lin_Wu:2014}. The same idea could be applied to different TF representation; for example, the SST based on wave packet transform or S-transform is recently considered in \cite{Yang:2014,Huang_Zhang_Zhao_Sun:2015}.

By carefully examining these methods, we see that there are several requirements a time series analysis method for an oscillatory signal should satisfy. First, if the signal is composed of several oscillatory components with different frequencies, the method should be able to decompose them. Second, if the oscillatory component has time-varying frequency or amplitude, then how the frequency or amplitude change should be well approximated. Third, if any of the oscillatory component exists only over a finite period, the algorithm should provide a clear information about the starting point and ending point. Fourth, if we represent the oscillatory behavior in the TF plane, then the TF representation should be sharp enough and contain the necessary information. Fifth, the algorithm should be robust to noise. Sixth, the analysis should be adaptive to the signal we want to analyze.
However, not every method could satisfy all these requirements. For example, due to the Heisenberg uncertainty principle, the TF representation of the STFT is blurred; the EMD is sensitive to noise and is incapable of handling the dynamics of the signal indicated in the third requirement. In addition to the above requirements, based on the problem we have interest, other features are needed from the TF analysis method, and some of them might not be easily fulfilled by the above approaches. 

Among these methods, SST \cite{Daubechies_Maes:1996,Daubechies_Lu_Wu:2011,Chen_Cheng_Wu:2014} and its variation \cite{Li_Liang:2012,Yang:2014,Huang_Zhang_Zhao_Sun:2015,Oberlin_Meignen_Perrier:2015} could simultaneously satisfies these requirements, but it still has limitations. While SST could analyze oscillatory signals of ``slowly varying instantaneous frequency (IF)'' well with solid mathematical supports, the window needs to be carefully chosen if we want to analyze signals with fast varying IF \cite{Li_Sheu_Laughlin_Chu:2015}. Precisely, the conditions $|A_k'(t)|\leq\epsilon\phi'_k(t)$ and $|\phi''_k(t)|\leq\epsilon\phi'_k(t)$ are essential if we want to study the model $f(t)=\sum_{k=1}^KA_k(t)\cos(2\pi\phi_k(t))$ by the current SST algorithm proposed in \cite{Daubechies_Maes:1996,Daubechies_Lu_Wu:2011,Chen_Cheng_Wu:2014}. 
Note that these ``needs'' could be understood/modeled as some suitable constraints, and to analyze the signal and simultaneously fulfill the designed constraints, optimization is a natural approach. Thus, in this paper, based on previous works and the above requirements, we would consider an optimization approach to study the oscillatory signals, which not only satisfies the above requirements, but also captures other features. In particular, we focus on capturing the fast varying IF. In brief, based on the relationship among the oscillatory components, the reconstruction property and the sparsity requirement on the time-frequency representation, we suggest to evaluate the optimal TF representation, denoted as $F$, by optimizing the following functional
\begin{align}
  \mathcal{H}(F,G):= &\, \int \left| \Re\int F(t,\omega) \ud\omega-f(t)\right|^2 \ud t\nonumber\\
   &\quad+\mu \iint |\partial_t F(t,\omega) -i2\pi \omega F(t,\omega) +G(t,\omega) \partial_\omega F(t,\omega)|^2 \ud t \ud\omega\label{OurFunctional}\\
       &\quad+\lambda\|F\|_{L^1}+\gamma\|G\|_{L^2},\nonumber
\end{align} 
where $G$ is an auxiliary function which quantifies the potentially fast varying instantaneous frequency. When $G$ is fixed, it is clear that although $\mathcal{H(\cdot,G)}$ is not strictly convex, it is convex, so finding the minimizer is guaranteed. To solve this optimization problem, we propose to apply the widely applied and well studied algorithm {\it Fast Iterative Shrinkage-Thresholding Algorithm (FISTA)}. Embedded in an alternating minimization approach to estimate $G$ and $F$, we coin the algorithm as {\it Time-frequency bY COnvex OptimizatioN} (Tycoon).

The paper is organized in the following way. In Section \ref{Section:AdaptiveHarmonicModel}, we discuss the adaptive harmonic model to model the signals with a fast varying instantaneous frequency and its identifiability problem; in Section \ref{Section:OptimizationApproach}, the motivation of the optimization approach based on the functional (\ref{OurFunctional}) is provided; in Section \ref{Section:NumericalImplementation}, we discuss the numerical details of Tycoon. In particular, how to apply the FISTA algorithm to solve the optimization problem; in Section \ref{Section:NumericalResults}, numerical results of Tycoon are provided. 

\section{Adaptive Harmonic Model}\label{Section:AdaptiveHarmonicModel}

We start from introducing the model which we use to capture the signal with ``fast varying IF''. 
The oscillatory signals with fast varying IF is commonly encountered in practice, for example, the chirp signal generated by bird's song, bat's vocalization and wolf's howl, the uterine electromyogram signal, the heart rate time series of a subject with atrial fibrillation, the gravitational wave and the vibrato in violin play or human voice. More examples could be found in \cite{Flandrin:2001}. Thus, finding a way to study this kind of signal is fundamentally important in data analysis.
First, we introduce the following model to capture the signals with fast varying IF, which generalizes the $\mathcal{A}^{c_1,c_2}_{\epsilon,d}$ class considered in \cite{Daubechies_Lu_Wu:2011,Chen_Cheng_Wu:2014}:
\begin{defn}[Generalized intrinsic mode type function (gIMT)] 
Fix constants $0\leq \epsilon\ll 1$, $c_2>c_1>\epsilon$ and $c_2>c_3>\epsilon$. Consider the functional set $\mathcal{Q}_{\epsilon}^{c_1,c_2,c_3}$, which consists of functions in $C^1(\RR)\cap L^\infty(\RR)$ with the following format: 
\begin{equation}
g(t) = A (t)\cos(2\pi\phi(t)),
\end{equation}
which satisfies the following {\em regularity conditions}
\begin{equation}\label{AHM:Condition:Regularity}
A\in C^1(\RR)\cap L^\infty(\RR),\quad\phi\in C^3(\RR),
\end{equation}
the {\em boundedness conditions} for all $t\in\RR$
\begin{align}
&\inf_{t\in\RR} A(t)\geq c_1,\quad \inf_{t\in\RR}\phi'(t)\geq c_1,\label{AHM:Condition:Boundedness}\\  
&\sup_{t\in\RR} A(t)\leq c_2,\quad\sup_{t\in\RR}\phi'(t)\leq c_2,\quad\sup_{t\in\RR}|\phi''(t)|\leq c_3,\nonumber
\end{align}
and the {\em growth conditions} for all $t\in\RR$
\begin{equation}\label{AHM:Condition:Growth}
|A'(t)|\leq \epsilon\phi'(t),\quad |\phi'''(t)|\leq \epsilon\phi'(t).
\end{equation}
\end{defn}

\begin{defn}[Adaptive harmonic model] 
Fix constants $0\leq \epsilon\ll 1$, $d>0$ and $c_2>c_1>0$. Consider the functional set $\mathcal{Q}_{\epsilon,d}^{c_1,c_2,c_3}$, which consists of functions in $C^1(\RR)\cap L^\infty(\RR)$ with the following format: 
\begin{equation}
g(t) = \sum_{\ell=1}^K g_\ell(t),
\end{equation}
where $K$ is finite and $g_\ell(t)=A_\ell(t)\cos(2\pi\phi_\ell(t))\in \mathcal{Q}_{\epsilon}^{c_1,c_2,c_3}$; when $K>1$, the following {\em separation condition} is satisfied:
\begin{equation}\label{definition:adaptiveHarmonicMultiple}
\phi_{\ell+1}'(t)-\phi'_\ell(t)>d
\end{equation}
for all $\ell=1,\ldots,K-1$. 
\end{defn}

We call $\epsilon,d,c_1,c_2$ and $c_3$ {\it model parameters} of the $\mathcal{Q}_{\epsilon,d}^{c_1,c_2,c_3}$ model. Clearly, $\mathcal{Q}_{\epsilon}^{c_1,c_2,c_3}\subset \mathcal{Q}_{\epsilon,d}^{c_1,c_2,c_3}$ and both $\mathcal{Q}_{\epsilon}^{c_1,c_2,c_3}$ and $\mathcal{Q}_{\epsilon,d}^{c_1,c_2,c_3}$ are not vector spaces. Note that in the $\mathcal{A}_{\epsilon,d}^{c_1,c_2}$ model, the condition ``$\phi_\ell\in C^3(\RR)$, $\sup_{t\in\RR}|\phi_\ell''(t)|\leq c_2$ and $|\phi_\ell'''(t)|\leq \epsilon\phi_\ell'(t)$ for all $t\in\RR$'' is replaced by ``$\phi_\ell\in C^2(\RR)$ and $|\phi_\ell''(t)|\leq \epsilon \phi_\ell'(t)$ for all $t\in\RR$''. Thus, we say that the signals in $\mathcal{A}_{\epsilon,d}^{c_1,c_2}$ are oscillatory with slowly varying instantaneous frequency. Also note that $\mathcal{A}_{\epsilon,d}^{c_1,c_2}$ is not a subset of $\mathcal{Q}_{\epsilon,d}^{c_1,c_2,c_3}$. Indeed, for $A_\ell(t)\cos(2\pi\phi_\ell(t))\in \mathcal{A}_{\epsilon,d}^{c_1,c_2}$, even if $\phi_\ell\in C^3(\RR)$, the third order derivative of $\phi_\ell$ is not controlled. Also note that the number of possible components $K$ is controlled by the model parameters; that is, $K\leq \frac{c_2-c_1}{d}$. 

\begin{rem} 
{\em We have some remarks about the model. First, note that it is possible to introduce more constants to control $A(t)$, like $0<c_4\leq \inf_{t\in\RR} A(t) \leq \sup_{t\in\RR} A(t)\leq c_5$, in addition to the control of $\phi'$ by $c_1,c_2>0$ in the model. Also, to capture the ``dynamics'', we could consider a more general model dealing with the ``sudden appearance/disappearance'', like $g(t) = \sum_{\ell=1}^K g_\ell(t)\chi_{I_\ell}$, where $\chi$ is the indicator function and $I_\ell\subset \RR$ is connected and long enough. However, while these will not generate fundamental differences but will complicate the notation, to simplify the discussion, we stick to our current model. 

Second, we could consider different models to study the ``fast varying IF''. For example, we could replace the condition ``$|A'(t)|\leq \epsilon\phi'(t)$, $\phi_\ell\in C^3(\RR)$, $\sup_{t\in\RR}|\phi_\ell''(t)|\leq c_2$ and $|\phi_\ell'''(t)|\leq \epsilon\phi_\ell'(t)$ for all $t\in\RR$'' by the {\em slow evolution chirp} conditions \cite{Flandrin:2001}; that is ``$|A'(t)|\leq \epsilon A(t)\phi'(t)$, $\phi_\ell\in C^2(\RR)$ and $|\phi_\ell''(t)|\leq \epsilon \phi_\ell'(t)^2$ for all $t\in\RR$''. We refer the reader with interest in the detailed discussion about this ``slow evolution chirp model'' to \cite[Section 2.2]{Flandrin:2001}. A simplified slow evolution chirp model (with the condition $|A'(t)|\leq \epsilon \phi'(t)$) is recently considered in \cite{Liu_Hou_Shi:2015} for the study of the sparsity approach to TF analysis. We mention that the argument about the identifiability issue stated below for $\mathcal{Q}_{\epsilon,d}^{c_1,c_2,c_3}$ could be directly applied to state the identifiability issue of the slow evolution chirp model.} 
\end{rem}

Before proceeding to say what it means by ``instantaneous frequency'' or ``amplitude modulation'', we immediately encounter a problem which is understood as the {\it identifiability problem}. Indeed, we might have infinitely many different ways to represent a cosine function $g_0(t)=\cos(2\pi t)$ in the format $a(t)\cos(2\pi\phi(t))$ so that $a>0$ and $\phi'>0$, even though it is well known that $g_0(t)$ is a harmonic function with amplitude $1$ and frequency $1$. Precisely, there exist infinitely many smooth functions $\alpha$ and $\beta$ so that $g_0(t)=\cos(2\pi t)=(1+\alpha(t))\cos(2\pi(t+\beta(t)))$, and in general there is no reason to favor $\alpha(t)=\beta(t)=0$. Before resolving this issue, we could not take amplitude $1$ and frequency $1$ as reliable features to quantify the signal $g_0$ when we view it as a component in $\mathcal{Q}_{\epsilon}^{c_1,c_2,c_3}$. In \cite{Chen_Cheng_Wu:2014}, it is shown that if $g(t)=A(t)\cos(2\pi\phi(t))=[A(t)+\alpha(t)]\cos(2\pi[\phi(t)+\beta(t)])$ are both in $\mathcal{A}_{\epsilon,d}^{c_1,c_2}$, then $|\alpha(t)|\leq C\epsilon$ and $|\beta'(t)|\leq C\epsilon$, where $C$ is a constant depending only on the model parameters $c_1,c_2,d$. Therefore, $A_\ell$ and $\phi'_\ell$ are unique locally up to an error of order $\epsilon$, and hence we could view them as features of an oscillatory signal in $\mathcal{A}_{\epsilon,d}^{c_1,c_2}$. Here, we show a parallel theorem describing the identifiability property for the functions in the $\mathcal{Q}_{\epsilon,d}^{c_1,c_2,c_3}$ model. 

\begin{thm}[Identifiability of $\mathcal{Q}^{c_1,c_2,c_3}_\epsilon$]\label{theorem:identifiability:single}
Suppose a gIMT $a(t)\cos\phi(t)\in \mathcal{Q}^{c_1,c_2,c_3}_\epsilon$ can be represented in a different form which is also a gIMT in $\mathcal{Q}^{c_1,c_2,c_3}_\epsilon$; that is, $a(t)\cos\phi(t)=A(t)\cos\varphi(t)\in \mathcal{Q}^{c_1,c_2,c_3}_\epsilon$. Define $t_m:=\phi^{-1}((m+1/2)\pi)$ and $s_m:=\phi^{-1}(m\pi)$, $m\in\ZZ$, $\alpha(t):=A(t)-a(t)$, and $\beta(t):=\varphi(t)-\phi(t)$. Then we have the following controls of $\alpha$ and $\beta$ at $t_m$ and $s_m$
\begin{enumerate}
\item Up to a global factor $2l\pi$, $l\in\ZZ$, $\beta(t_n)=0$ for all $n\in\ZZ$;
\item $\frac{a(t_n)}{a(t_n)+\alpha(t_n)}=\frac{\phi'(t_n)+\beta'(t_n)}{\phi'(t_n)}$ for all $n\in\ZZ$. In particular, $\alpha(t_n)=0$ if and only if $\beta'(t_n)=0$ for all $n\in\ZZ$;
\item $\frac{a(s_n)}{a(s_n)+\alpha(s_n)}=\cos(\beta(s_n))$ for all $n\in\ZZ$. In particular, $\alpha(s_m)=0$ if and only if $\beta(s_m)=0$, $m\in\ZZ$.
\end{enumerate}
Furthermore, the size of $\alpha$ and $\beta$ are bounded by
\begin{enumerate}
\item $|\alpha(t)|< 2\pi \epsilon$ for all $t\in\RR$;
\item $|\beta''(t)|\leq 2\pi\epsilon$, $|\beta'(t)|\leq \frac{2\pi\epsilon}{c_1}$ and $|\beta(t)|\leq \frac{2\pi\epsilon}{c^2_1}$ up to a global factor $2l\pi$, $l\in\ZZ$, for all $t\in\RR$.
\end{enumerate}
\end{thm}

We mention that the controls of $\alpha$ and $\beta$ at $t_m$ and $s_m$ do not depend on the growth condition in (\ref{AHM:Condition:Growth}). However, to control the size of $\alpha$ and $\beta$, we need the growth condition in (\ref{AHM:Condition:Growth}).

\begin{thm}[Identifiability of $\mathcal{Q}^{c_1,c_2,c_3}_{\epsilon,d}$]\label{theorem:identifiability:multiple}
Suppose $f(t)\in \mathcal{Q}^{c_1,c_2,c_3}_{\epsilon,d}$ can be represented in a different form which is also in $\mathcal{Q}^{c_1,c_2,c_3}_{\epsilon,d}$; that is, 
\begin{align}
f(t)=\sum_{l=1}^Na_l(t)\cos\phi_l(t)=\sum_{l=1}^MA_l(t)\cos\varphi_l(t)\in \mathcal{Q}^{c_1,c_2,c_3}_{\epsilon,d}.
\end{align}
Then, when $d\geq \sqrt{2\ln c_2+\frac{1}{2}\ln c_3-\ln \epsilon}$, $M=N$ and for all $t\in\RR$ and for all $l=1,\ldots,N$, the following holds:
\begin{enumerate}
\item $|\phi_l(t)-\varphi_l(t)|=O(\sqrt{\epsilon})$ up to a global factor $2n\pi$, $n\in\ZZ$;
\item $|\phi'_l(t)-\varphi'_l(t)|=O(\sqrt{\epsilon})$;
\item $|\phi''_l(t)-\varphi''_l(t)|=O(\sqrt{\epsilon})$;
\item $|a_l(t)-A_l(t)|= O(\sqrt{\epsilon})$,
\end{enumerate}
where the constants on the right hand side are universal constants depending on the model parameters of $\mathcal{Q}_{\epsilon,d}^{c_1,c_2,c_3}$.
\end{thm}
Note that in this theorem, the bound $\sqrt{\epsilon}$ and the lower bound of $d$ are by no means optimal since we consider the case when there are as many components as possible. We focus on showing that even when there are different representations of a given function in $\mathcal{Q}_{\epsilon,d}^{c_1,c_2,c_3}$, the quantities we have interest are close up to a negligible constant.
As a result, we have the following definitions, which generalize the notion of amplitude and frequency. 

\begin{defn}{[Phase function, instantaneous frequency, chirp factor and amplitude modulation]} 
Take a function $f(t)=\sum_{\ell=1}^Na_\ell(t)\cos\phi_\ell(t)\in\mathcal{Q}_{\epsilon,d}^{c_1,c_2,c_3}$. For each $\ell=1,\ldots,N$, the monotonically increasing function $\phi_\ell(t)$ is called the {\it phase function} of the $\ell$-th gIMT; the first derivative of the phase function, $\phi_\ell'(t)$, is called the {\it instantaneous frequency} (IF) of the $\ell$-th gIMT; the second derivative of the phase function, $\phi_\ell''(t)$, is called the {\it chirp factor} (CF) of the $\ell$-th gIMT; the positive function $A_\ell(t)$ is called the {\it amplitude modulation} (AM) of the $\ell$-th gIMT. 
\end{defn}
Note that the IF and AM are always positive, but usually not constant. On the other hand, the CF might be negative and non-constant. Clearly, when $\phi_\ell$ are all linear functions with positive slopes and $A_\ell$ are all positive constants, then the model is reduced to the harmonic model and the IF is equivalent to the notion frequency in the ordinary Fourier transform sense. The conditions $|A'_\ell(t)|\leq \epsilon \phi'_\ell(t)$ and $|\phi'''_\ell(t)|\leq \epsilon\phi_\ell'(t)$ force the signal to locally behave like a harmonic function or a chirp function, and hence the nominations.  By Theorem \ref{theorem:identifiability:single} and Theorem \ref{theorem:identifiability:multiple}, we know that the definition of these quantities are unique up to an error of order $\epsilon$. 

We could also model the commonly encountered ingredient in signal processing -- the shape function, trend and noise as those considered in \cite{Wu:2013,Chen_Cheng_Wu:2014}. However, to concentrate the discussion on the optimization approach to the problem, in this paper we focus only on the $\mathcal{Q}_{\epsilon,d}^{c_1,c_2,c_3}$ functional class.

\section{Optimization Approach}\label{Section:OptimizationApproach}
In general, given a function $f(t)=\sum_{k=1}^KA_k(t)\cos(2\pi\phi_k(t))$ so that $A_k(t)>0$ and $\phi'_k(t)>0$ for $t\in\RR$, we would expect to have the {\it ideal time-frequency representation} (iTFR), denoted as $R_f(t,\omega)$, satisfying 
\begin{align}
R_f(t,\omega)=\sum_{k=1}^KA_k(t)e^{i2\pi\phi_k(t)}\delta_{\phi'_k(t)}(\omega),
\end{align} 
where $\delta_{\phi'_k(t)}$ is the Dirac measure supported at $\phi'_k(t)$, so that we could well extract the features $A_k(t)$ and $\phi'_k(t)$ describing the oscillatory signal from $R_f$. Note that the iTFR is a distribution. In addition, the reconstruction and visualization of each component are possible. Indeed, we can reconstruct the $k$-th component by integrating along the frequency axis on the period near $\phi'_k(t)$. Indeed,
\begin{align}
A_k(t)\cos(2\pi\phi_k(t))=\Re\int_\RR R_f(t,\omega)\psi\left(\frac{\omega-\phi'_k(t)}{\theta}\right) \ud\omega,
\end{align}
where $\Re$ means taking the real part, $\theta\ll1$, $\psi$ is a compactly supported Schwartz function so that $\psi(0)=1$. Further, the visualization is realized via displaying the ``time-varying power spectrum'' of $f$, which is defined as 
\begin{align}
S_f(t,\omega):=\sum_{k=1}^KA^2_k(t)\delta_{\phi'_k(t)}(\omega),
\end{align} 
and we call it the {\it ideal time-varying power spectrum} (itvPS) of $f$, which is again a distribution.

To evaluate the iTFR for a function $f=\sum_{k=1}^KA_k(t)\cos(2\pi\phi_k(t))$, 
we fix $0<\theta\ll 1$ and consider the following {\it approximative iTFR with resolution $\theta$} 
\begin{align}
\tilde{R}_f(t,\omega)=\sum_{k=1}^KA_k(t)e^{i2\pi\phi_k(t)}\frac{1}{\theta}h\left(\frac{\omega-\phi'_k(t)}{\theta}\right),
\end{align}
where $t\in\RR$,
$\omega\in\RR$ and $h$ is a Schwartz function supported on $[-\sigma,\sigma]$, $\sigma>0$, so that $\int h=1$ and $\frac{1}{\epsilon}h\left(\frac{\cdot}{\epsilon}\right)$ converges to Dirac measure $\delta$ supported at $0$ weakly as $\epsilon\to 0$ and $\int h(x)\ud x=1$. Clearly, we know that $\tilde{R}_f$ is essentially supported around $(t,\phi'_k(t))$ for $k=1,\ldots,K$ and as $\theta\to 0$, $\tilde{R}_f$ converges to the iTFR in the weak sense. Also, we have for all 
$t\in\RR$ 
and $k=1,\ldots,K$, when $\theta$ is small enough so that $\sigma\theta>d$ is satisfied, where $d$ is the constant defined in the separation condition in (\ref{definition:adaptiveHarmonicMultiple}), we have
\begin{align}
\Re\int_{\phi'_k(t)-\sigma\theta}^{\phi'_k(t)+\sigma\theta} \tilde{R}_f(t,\omega) \ud \omega = A_k(t)\cos(2\pi\phi_k(t)).
\end{align}
Thus, the reconstruction property of iTFR is satisfied. In addition, the visualization property of itvPS can be achieved by taking
\begin{align}
\tilde{S}_f(t,\omega)=\left|\tilde{R}_f(t,\omega)\right|^2=\sum_{k=1}^K|A_k(t)|^2\frac{1}{\theta^2}\left|h\left(\frac{\omega-\phi'_k(t)}{\theta}\right)\right|^2,
\end{align}
where the equality holds due to the facts that $\phi'_k$ are separated and $\theta\ll 1$. 
Next we need to find other conditions about $\tilde{R}_f$. A natural one is observing its differentiation.
By a direct calculation, we know $\frac{1}{\theta^2}h'\left(\frac{\omega-\phi'_k(t)}{\theta}\right)= \partial_\omega\frac{1}{\theta}h\left(\frac{\omega-\phi'_k(t)}{\theta}\right)$, and hence we have
\begin{align}
\partial_t\tilde{R}_f(t,\omega)=&\sum_{k=1}^K A_k'(t)e^{i2\pi\phi(t)}\frac{1}{\theta}h\left(\frac{\omega-\phi'_k(t)}{\theta}\right) \\
&+ i2\pi\sum_{k=1}^KA_k(t)\phi_k'(t)e^{i2\pi\phi_k(t)}\frac{1}{\theta}h\left(\frac{\omega-\phi'_k(t)}{\theta}\right)\nonumber\\
& -\sum_{k=1}^K A_k (t)e^{i2\pi\phi(t)}\phi_k''(t)\frac{1}{\theta^2}h'\left(\frac{\omega-\phi'_k(t)}{\theta}\right) \nonumber \\
=&\sum_{k=1}^K A_k'(t)e^{i2\pi\phi(t)}\frac{1}{\theta}h\left(\frac{\omega-\phi'_k(t)}{\theta}\right) \nonumber\\
&+ i2\pi\sum_{k=1}^KA_k(t)\phi_k'(t)e^{i2\pi\phi_k(t)}\frac{1}{\theta}h\left(\frac{\omega-\phi'_k(t)}{\theta}\right)\nonumber\\
& +\partial_\omega\sum_{k=1}^K A_k (t)e^{i2\pi\phi(t)}\phi_k''(t)\frac{1}{\theta}h\left(\frac{\omega-\phi'_k(t)}{\theta}\right).\nonumber
\end{align}
By the fact that  $\omega \tilde{R}_f(t,\omega)= \sum_{k=1}^KA_k(t)\omega e^{i2\pi\phi_k(t)}\frac{1}{\theta}h\left(\frac{\omega-\phi'_k(t)}{\theta}\right)$,  we have
\begin{align}
&\partial_t\tilde{R}_f(t,\omega)-i2\pi\omega \tilde{R}_f(t,\omega)\label{OptimizationMotivationEquation1} \\
=& \sum_{k=1}^K A_k'(t)e^{i2\pi\phi(t)}\frac{1}{\theta}h\left(\frac{\omega-\phi'_k(t)}{\theta}\right) \nonumber\\
&\quad- i2\pi\sum_{k=1}^KA_k(t)(\omega-\phi_k'(t) )e^{i2\pi\phi_k(t)}\frac{1}{\theta}h\left(\frac{\omega-\phi'_k(t)}{\theta}\right)\nonumber\\
&\quad +\partial_\omega\sum_{k=1}^K A_k (t)e^{i2\pi\phi_k(t)}\phi_k''(t)\frac{1}{\theta}h\left(\frac{\omega-\phi'_k(t)}{\theta}\right)\nonumber. 
\end{align}

We first discuss the case when $f\in\mathcal{A}_{\epsilon,d}^{c_1,c_2}$; that is, $|\phi''_k(t)|\leq \epsilon |\phi'_k(t)|$ for all $t\in \RR$.
Note that by the assumption of frequency separation (\ref{definition:adaptiveHarmonicMultiple}) and the fact that $\theta\ll 1$, $[\phi'_l(t)-\theta \sigma,\phi'_l(t)+\theta \sigma]\cap [\phi'_k(t)-\theta \sigma,\phi'_k(t)+\theta \sigma]=\emptyset$ when $l\neq k$. Thus we have 
\begin{align}
\left|\sum_{k=1}^K A_k'(t)e^{i2\pi\phi_k(t)}\frac{1}{\theta}h\left(\frac{\omega-\phi'_k(t)}{\theta}\right) \right|^2=\sum_{k=1}^K |A_k'(t)|^2 \frac{1}{\theta^2}h^2\left(\frac{\omega-\phi'_k(t)}{\theta}\right).
\end{align}
Indeed, when $\omega\in [\phi'_l(t)-\theta \sigma,\phi'_l(t)+\theta \sigma]$, we have
\begin{align}
\left|\sum_{k=1}^K A_k'(t)e^{i2\pi\phi(t)}\frac{1}{\theta}h\left(\frac{\omega-\phi'_k(t)}{\theta}\right) \right|^2=   |A_l'(t)|^2 \frac{1}{\theta^2}h^2\left(\frac{\omega-\phi'_l(t)}{\theta}\right).
\end{align}
The same argument holds for the other terms on the right hand side of (\ref{OptimizationMotivationEquation1}). As a result, by a direct calculation, for any non-empty finite interval $I\subset \RR$, we have
\begin{align}\label{OptimizationMotivationEquationSlow}
&\left\|\sqrt{\theta}\left(\partial_t\tilde{R}_f(t,\omega)-i2\pi\omega \tilde{R}_f(t,\omega)\right)\right\|^2_{L^2(I\times [0,\infty))}\\
\leq&\,    \left(\epsilon^2 J_{0,0,2}+2\pi\theta \epsilon J_{1,0,2} +4\pi^2\theta^2J_{2,0,2}+\frac{\epsilon^2c_2^2}{\theta^2}J_{0,1,2} \right)c_2^2I \nonumber,
\end{align}
where $J_{n,m,l}:=\int \eta^n [\partial_\eta^mh(\eta)]^l\ud \eta$, where $n,m,l=0,1,\ldots$. Thus, when $\epsilon$ is small enough, $\left\|\sqrt{\theta}\left(\partial_t\tilde{R}_f(t,\omega)-i2\pi\omega \tilde{R}_f(t,\omega)\right)\right\|^2_{L^2(I\times [0,\infty))}$ is small. Here, we mention that as the dynamic inside the signal we have interest is ``momentary'', we would expect to have a small error between $\partial_t\tilde{R}_f(t,\omega)$ and $i2\pi\omega \tilde{R}_f(t,\omega)$ ``locally'', which however, might accumulate when $I$ becomes large.
This observation leads to a variational approach discussed in \cite{Daubechies_Lu_Wu:2011}. Precisely, the authors in \cite{Daubechies_Lu_Wu:2011} considered to minimize the following functional when the signal $f$ is observed on a non-empty finite interval $I$:
   \begin{align}
 \mathcal{H}_0(F):=  & \int_I \left| \Re\int F(t,\omega) d\omega-f(t)\right|^2 \ud t\\
   &\qquad+\mu\iint_I \left|\partial_t F(t,\omega)-i2\pi\omega F(t,\omega)\right|^2 \ud t \ud \omega.\nonumber
  \end{align}
The optimal $F$ would be expected to approximate the iTFR of $f\in \mathcal{A}_{\epsilon,d}^{c_1,c_2,c_3}$ well. However, that optimization was not numerically carried out in \cite{Daubechies_Lu_Wu:2011}.

Now we come back to the case we have interest; that is, $f\in\mathcal{Q}_{\epsilon,d}^{c_1,c_2,c_3}$. 
Since the condition on the CF terms, that is, $|\phi''_k(t)|\leq \epsilon |\phi'_k(t)|$, no longer holds, the above bound (\ref{OptimizationMotivationEquationSlow}) does not hold and minimizing the functional $\mathcal{H}_0$ might not lead to the right solution. In this case, however, we still have the following bound by the same argument as that of (\ref{OptimizationMotivationEquationSlow}):
\begin{align}
&\left\|\sqrt{\theta} \left(\partial_t \tilde{R}_f(t,\omega) -i2\pi \omega \tilde{R}_f(t,\omega) -\partial_\omega\sum_{k=1}^K A_k (t)e^{i2\pi\phi_k(t)}\phi_k''(t)\frac{1}{\theta}h\left(\frac{\omega-\phi'_k(t)}{\theta}\right)\right)\right\|^2_{L^2(I\times [0,\infty))}\nonumber\\
\leq&\,\left\|\sqrt{\theta}\sum_{k=1}^K \big(A_k'(t)- i2\pi A_k(t)(\omega-\phi_k'(t) )\big)e^{i2\pi\phi(t)}\frac{1}{\theta}h\left(\frac{\omega-\phi'_k(t)}{\theta}\right) \right\|^2_{L^2(I\times [0,\infty))}\\
\leq&\,   \left(\epsilon^2 J_{0,0,2}+2\pi \epsilon\theta J_{1,0,2} +4\pi^2\theta^2J_{2,0,2}\right)c_2^2I,\nonumber
\end{align}  
Thus, once we find a way to express the extra term $\partial_\omega\sum_{k=1}^K A_k (t)e^{i2\pi\phi_k(t)}\phi_k''(t)\frac{1}{\theta}h\left(\frac{\omega-\phi'_k(t)}{\theta}\right)$ in a convenient formula, we could introduce another conditions on $F$. 

In the special case when $K=1$; that is, $f=A(t)\cos(2\pi\phi(t))$, we know that 
\begin{align}\label{OptimizationMotivationEquation2}
\partial_\omega \left[A (t)e^{i2\pi\phi(t)}\phi''(t)\frac{1}{\theta}h\left(\frac{\omega-\phi'(t)}{\theta}\right)\right]=\phi''(t)\partial_\omega \tilde{R}_f(t,\omega).
\end{align}
Thus, we have   
\begin{align}
\theta \iint_I |\partial_t \tilde{R}_f(t,\omega) -i2\pi \omega \tilde{R}_f(t,\omega) +\phi''(t) \partial_\omega \tilde{R}_f(t,\omega)|^2 \ud t \ud\omega=O(\theta^2,\theta\epsilon,\epsilon^2).
\end{align}
Thus, we could consider the following functional
\begin{align}
\theta \iint |\partial_t F(t,\omega) -i2\pi \omega F(t,\omega) +\alpha(t) \partial_\omega F(t,\omega)|^2 \ud t \ud\omega,
\end{align}
where $\alpha(t)\in\RR$ is used to capture the CF term associated with the ``fast varying instantaneous frequency''. Thus, when $K=1$, we can capture more general oscillatory signals by considering the following functional when the signal is observed on a non-empty finite interval $I\subset \RR$:
\begin{align}
  \mathcal{H}(F,\alpha):= &\, \int_I \left| \Re\int F(t,\omega) \ud\omega-f(t)\right|^2 \ud t\nonumber\\
   &\quad+\mu \theta\iint_I |\partial_t F(t,\omega) -i2\pi \omega F(t,\omega) +\alpha(t) \partial_\omega F(t,\omega)|^2 \ud t \ud\omega \label{functional3}\\
    &\quad+\lambda\|F\|_{L^1(I\times \RR)}+\gamma\|\alpha\|_{L^2(I\times \RR)},\nonumber
\end{align} 
where $F\in L^2(I\times \RR)$ is the function defined on the TF plane restricted on $I\times\RR$. 
Note that the $L^1$ norm is another constraint we introduce in order to enhance the sharpness of the TF representation. Indeed, we would expect to introduce a sparse TF representation when the signal is composed of several gIMT.

In general when $K>1$, we cannot link $\partial_\omega\sum_{k=1}^K A_k (t)e^{i2\pi\phi_k(t)}\phi_k''(t)\frac{1}{\theta}h\left(\frac{\omega-\phi'_k(t)}{\theta}\right)$ to $\partial_\omega \tilde{R}_f(t,\omega)$ by any function on $t$ like that in (\ref{OptimizationMotivationEquation2}). In this case, we could expect to find another function $G\in L^2(I\times \RR)$ so that 
\begin{equation}
G(t,\omega)=\left\{
\begin{array}{ll}
\phi_k''(t) & \mbox{ when }\omega\in[\phi_k'(t)-\theta\sigma,\phi'_k(t)+\theta\sigma]\\
0&\mbox{ otherwise}.
\end{array}
\right.
\end{equation}
and hence $G(t,\omega)\partial_\omega \tilde{R}_f(t,\omega)=\partial_\omega\sum_{k=1}^K A_k (t)e^{i2\pi\phi_k(t)}\phi_k''(t)\frac{1}{\theta}h\left(\frac{\omega-\phi'_k(t)}{\theta}\right)$. 
Thus, we could consider minimizing the following functional for a given function $f$ observed on a non-empty finite interval $I\subset \RR$:
\begin{align}
  \mathcal{H}(F,G):= &\, \int \left| \Re\int F(t,\omega) \ud\omega-f(t)\right|^2 \ud t\nonumber\\
   &\quad+\mu\theta\iint_I |\partial_t F(t,\omega) -i2\pi \omega F(t,\omega) +G(t,\omega) \partial_\omega F(t,\omega)|^2 \ud t \ud\omega \label{functional4}\\ 
       &\quad+\lambda\|F\|_{L^1(I\times \RR)}+\frac{\gamma}{\sqrt{\theta}}\|G\|_{L^2(I\times \RR)}.\nonumber
\end{align} 
Here, the $L^2$ penalty term $\|G\|_{L^2}$ has $1/\sqrt{\theta}$ in front of it since 
\begin{equation}
\|G\|_{L^2(I\times [0,\infty))}=\sqrt{2\theta\sigma}\sum_{k=1}^K\|\phi_k''\|_{L^2(I\times [0,\infty))}.
\end{equation}
Thus, the $L^2$ penalty term does not depend on $\theta$.
It is also clear that the $L^1$ penalty term in the above functional does not depend on $\theta$ as we have
\begin{equation}
\int\int_I\left|\sum_{k=1}^K A_k (t)e^{i2\pi\phi(t)}\frac{1}{\theta}h\left(\frac{\omega-\phi'_k(t)}{\theta}\right) \right|\ud t\ud \omega=\sum_{k=1}^K\|A_k(t)\|_{L^1(I)}.
\end{equation}

\section{Numerical Algorithm}\label{Section:NumericalImplementation}

We consider the following functionals associated with (\ref{functional3}):

\begin{align}
\mathcal{H}(F,\alpha) & = \int_{\mathbb{R}}\left|\Re \int_{\mathbb{R}}F(t,\omega)d\omega -f(t)\right|^{2}\ud t\\
+ & \tilde \mu \left( \tilde \lambda \iint_{\mathbb{R}}\left|\partial_t F(t,\omega) - i2\pi\omega F(t,\omega)+\alpha(t) \partial_\omega F(t,\omega)\right|^{2}\ud t\ud\omega 
+ (1-\tilde \lambda)\|F\|_{L^1} \right) 
+ \gamma\|\alpha\|_{L^2}^2 \nonumber\\
 			  & = \mathcal{G}(F,\alpha) + \Psi(F,\alpha)\nonumber,
\end{align}
where	
\begin{align}
\mathcal{G}(F,\alpha):=& \int_{\mathbb{R}}\left|\Re \int_{\mathbb{R}}F(t,\omega)d\omega -f(t)\right|^{2}\ud t\\
&+\tilde \mu \tilde \lambda \iint_{\mathbb{R}}\left|\partial_t F(t,\omega) - i2\pi\omega F(t,\omega)+\alpha(t) \partial_\omega F(t,\omega)\right|^{2}\ud t\ud\omega,\nonumber
\end{align}
\begin{align}
\Psi(F,\alpha):=\tilde \mu(1-\tilde \lambda)\|F\|_{L^1}+\gamma\|\alpha\|_{L^2}^2,
\end{align}
$t$ is the time and $\omega$ is the frequency. The numerical implementation of (\ref{functional4}) follows the same lines while we have to discretize a two dimensional function $G$. 
Compared to \eqref{functional3}, we have redefined the role of the hyperparameter $\mu$ and $\lambda$. Here, $\tilde \mu\in\mathbb{R}_+$ balance between the data fidelity term $\int_{\mathbb{R}}\left|\Re \int_{\mathbb{R}}F(t,\omega)d\omega -f(t)\right|^{2}\ud t$ allowing the reconstruction, and the regularization term which controls the variation on the derivatives, and the sparsity of the solution. The parameter $\tilde \lambda \in [0,1]$ allows one to balance between the sparsity prior and the constraint of the derivatives. This choice will simplify the choice of the regularization parameters.
Clearly, by setting $\mu = \tilde\mu\tilde \lambda$ and $\lambda = \tilde \mu(1-\tilde\lambda) $ we recover the original formulation \eqref{functional3}.

\subsection{Numerical discretization}
Numerically, we consider the following discretization of $F$ by taking $\Delta_t>0$ and $\Delta_\omega>0$ as the sampling periods in the time axis and frequency axis. We also restrict $F$ to time $[0, M\Delta_t]$ and 
to the frequencies $[-N\Delta_\omega, N\Delta_\omega]$.
Then, we discretize $F$ as $\boldsymbol{F}\in \mathbb{C}^{(N+1)\times(M+1)}$ and $\alpha$ as $\valpha\in \mathbb{R}^{M+1}$, where
\begin{align}
\boldsymbol{F}_{n,m}=F(t_{m},\omega_{n}),\quad \valpha_m=\alpha(t_m),
\end{align}
$t_m:=m\Delta_t$, $\omega_n:=n\Delta_\omega$, $n=-N,\ldots,N$ and $m=0,1,\ldots, M$. The observed signal $f(t)$ is discretized as a $(M+1)$-dim vector $\boldsymbol{f}$, where
\begin{align}
\boldsymbol{f}_l=f(t_l).
\end{align}
Note that the sampling period of the signal $\Delta_t$ and $M$ most of time are determined by the data collection procedure. We could set $\Delta_\omega=\frac{1}{M\Delta_t}$ and $N=\lceil M/2\rceil$ suggested by the Nyquist rate in the sampling theory.  

Next, using the rectangle method, we could discretize $\mathcal{G}(F,\alpha)$ directly by  
\begin{align}
\mathcal{G}(\boldsymbol{F},\valpha):=&\sum\limits_{m=0}^{M}\left|\sum\limits_{n=-N}^{N}2\Re\left(F(t_{m},\omega_{n})\right)\Delta_\omega-f(t_{m})\right|^{2}\Delta_t \\
&+\mu\sum\limits_{m=0}^{M}\sum\limits_{n=-N}^{N}\left|\partial_{t}F(t_{m},\omega_{n})-i2\pi\omega_{n}F(t_{m},\omega_{n})+\alpha(t_m)\partial_{\omega}F(t_{m},\omega_{n})\right|^{2}\Delta_t\Delta_\omega .\nonumber
\end{align}
The partial derivative $\partial_{t}F$ can be implemented by the straight finite difference; that is, take a $(M+1)\times(M+1)$ finite difference matrix $\boldsymbol{D}_{M+1}$ so that $\boldsymbol{F}\boldsymbol{D}_{M+1}$ approximates the discretization of $\partial_{t}F$. However, this choice may lead to numerical instability. Instead, one can implement the partial derivative in the Fourier domain, using that 
$\partial_{t} F(t_m,\omega_n) = \mathcal{F}^{-1}\left(i2\pi\xi_k \hat F(\xi_k,\omega_n)\right)[m]$, where $\hat F = \mathcal{F}(F)$ and $\mathcal{F}$ denotes the finite Fourier transform. 
For the sake of simplicity, we still denote by $\partial_t$ or $\partial_\omega$ the discretization operator in the discret domain, whatever the chosen method (finite difference or in the Fourier domain). 
Also denote $\boldsymbol{1}=(1,\ldots,1)^T\in \mathbb{R}^{M+1}$. In the matrix form, the functional $\mathcal{G}(F,\alpha)$ is thus discretized as
\begin{align}
\mathcal{G}(\vF,\valpha)=\Delta_t\left\|\mathcal{A}\vF-\vF\right\|^{2}+\Delta_t\Delta_\omega\mu\left\|\mathcal{B}(\vF,\valpha)\right\|^{2},
\end{align}
where
\begin{align}
\begin{array}{cccc}
\mathcal{A}  : & \mathbb{C}^{(M+1)\times(M+1)} & \to & \mathbb{R}^{M+1} \\
  & \vF & \mapsto & 2\Re\left(\boldsymbol{1}^T\vF \right)\Delta_\omega\,,\\
\end{array}
\end{align}
\begin{align}
\begin{array}{ccccc}
\mathcal{B}  : & \mathbb{C}^{(M+1)\times(M+1)}\times \CC^{M+1} & \to & \mathbb{C}^{(N+1)\times(M+1)} \\
  & (\vF,\valpha) & \mapsto & \partial_t\vF -i2\pi\boldsymbol{\omega}\vF+\partial_\omega\vF\text{diag}(\valpha)\,, \\
\end{array}
\end{align}
and $\boldsymbol{\omega}=\text{diag}(-N\Delta_\omega,\ldots,0,\Delta_\omega,2\Delta_\omega,\ldots,N\Delta_\omega)\in \mathbb{R}^{(M+1)\times(M+1)}$.

\subsection{Expression of the gradient operator}
Denote $\mathcal{G}_{\valpha}(\vF):= \vF \mapsto \mathcal{G}(\vF,\valpha)$ and $\mathcal{B}_{\valpha}(\vF):= \vF \mapsto
 \mathcal{B}(\vF,\valpha)$; that is, $\valpha$ is fixed. Similarly, define $\mathcal{G}_{\vF}(\valpha):=\valpha \mapsto 
 \mathcal{G}(\vF,\valpha)$ and $\mathcal{B}_{\vF}(\valpha):=\valpha \mapsto \mathcal{B}(\vF,\valpha)$; that is, $\vF$ is fixed. 
 We will evaluate the gradient of $\mathcal{G}_{\valpha}$ and $\mathcal{G}_{\vF}$ after discretization for the gradient decent 
 algorithm. Take $\vG\in\CC^{(M+1)\times(M+1)}$.
The gradient of $\mathcal{G}_{\valpha}$ after discretization is evaluated by
\begin{align}
\nabla\mathcal{G}_{\valpha}|_{\vF}\vG&=\,\lim_{h\to 0}\frac{\mathcal{G}_{\valpha}(\vF+h\vG)-\mathcal{G}_{\valpha}(\vF)}{h}\\
&=\,2\Delta_t(\mathcal{A}\vF-\vf)^T\mathcal{A}\vG+2\Delta_t\Delta_\omega\mu\langle\mathcal{B}_{\valpha}\vF, \mathcal{B}_{\valpha}\vG\rangle\nonumber\\
&=\,\langle 2\Delta_t\mathcal{A}^*(\mathcal{A}\vF-\vf)+2\Delta_t\Delta_\omega\mu\mathcal{B}_{\valpha}^*\mathcal{B}_{\valpha}\vF,\vG\rangle.\nonumber
\end{align}
As a result, we have
\begin{align}
\nabla\mathcal{G}_{\valpha}|_{\vF}= 2\Delta_t\mathcal{A}^*(\mathcal{A}\vF-\vf)+2\Delta_t\Delta_\omega\mu\mathcal{B}_{\valpha}^*\mathcal{B}_{\valpha}\vF.
\end{align}
where $\mathcal{A}^{*}$ and $\mathcal{B}_{\valpha}^{*}$ are adjoint operators of $\mathcal{A}$ and $\mathcal{B}_{\valpha}$ respectively.
Now we expand $\mathcal{A}^{*}$ and $\mathcal{B}_{\valpha}^{*}$. 
Take $\vg\in \mathbb{R}^{M+1}$. We have
\begin{align}
\left\langle \mathcal{A}\vF, \vg\right\rangle &= \sum_{m=0}^{M}\left(\sum_{n=-N}^{N}2\Re\vF_{n,m}\Delta_\omega\right)\vg_m\\
&=\sum_{m=0}^{M}\sum_{n=-N}^{N}2\Re\vF_{n,m}\Re(\Delta_\omega \vg_m),\nonumber
\end{align}
and
\begin{align}
\left\langle \vF, \mathcal{A}^{*}\vg\right\rangle=&\,\sum_{m=0}^{M}\sum_{n=-N}^{N}\vF_{n,m}\overline{(\mathcal{A}^{*}\vg)}_{n,m}\\
=&\,\sum_{m=0}^{M}\sum_{n=-N}^{N}2\Re\vF_{n,m}\Re(\mathcal{A}^{*}\vg)_{n,m}+\sum_{m=0}^{M}\sum_{n=-N}^{N}\Im\vF_{n,m}\Im(\mathcal{A}^{*}\vg)_{n,m}\nonumber\\
&+i\sum_{m=0}^{M}\sum_{n=-N}^{N}\Im\vF_{n,m}2\Re(\mathcal{A}^{*}\vg)_{n,m}-i\sum_{m=0}^{M}\sum_{n=-N}^{N}\Re\vF_{n,m}\Im(\mathcal{A}^{*}\vg)_{n,m}.\nonumber
\end{align}
Since $\left\langle \mathcal{A}\vF, \vg\right\rangle=\left\langle \vF, \mathcal{A}^{*}\vg\right\rangle$ for all $\vF$ and $\vg$, we conclude that
\begin{equation}
\begin{array}{ccccc}\label{formula:Astar}
\mathcal{A}^{*}  : & \mathbb{R}^{M+1} & \to & \mathbb{C}^{(M+1)\times(M+1)} \\
  & \vg & \mapsto & 2\Delta_\omega
	\begin{pmatrix}
   \vg_1 &\ldots& \vg_{M+1}\\
   \vdots&&\vdots\\
	 \vg_1&\ldots& \vg_{M+1}
	\end{pmatrix}
	\\
\end{array}.
\end{equation}
To calculate $\mathcal{B}_{\valpha}^{*}$, by a direct calculation we have
\begin{align}
\langle \mathcal{B}_{\valpha}\vF, \vG\rangle&=\langle \partial_t\vF -i2\pi\boldsymbol{\omega}\vF+\partial_\omega\vF\text{diag}(\valpha), \vG\rangle\\
&=\langle \vF, -\partial_t\vG+i2\pi\boldsymbol{\omega}\vG-\partial_\omega\vG\text{diag}(\valpha) \rangle\nonumber\\
&=\left\langle \vF, \mathcal{B}_{\valpha}^{*}\vG\right\rangle,\nonumber
\end{align}
where $\vG\in\mathbb{C}^{(M+1)\times(M+1)}$. Thus, we conclude that
\begin{equation}\label{formula:Bstar}
\begin{array}{ccccc}
\mathcal{B}_{\valpha}^{*}  : & \mathbb{C}^{(M+1)\times(M+1)} & \to & \mathbb{C}^{(M+1)\times(M+1)} \\
  & \vG & \mapsto &  -\partial_t\vG+i2\pi\boldsymbol{\omega}\vG-\partial_\omega\vG\text{diag}(\valpha).
\end{array}
\end{equation}
As a result, the first part of $\nabla\mathcal{G}_{\valpha}|_{\vF}$, $2\Delta_t\mathcal{A}^*(\mathcal{A}\vF-\vf)$, can be numerically expressed as 
\begin{equation}
	\label{eq:grad1}
4\Delta_t\Delta_\omega
\begin{pmatrix}
\Delta_\omega\Re \sum\limits_{n=0}^{N}\vF_{n,1}-\vf_{1}&\ldots& \Delta_\omega\Re \sum\limits_{n=0}^{N}\vF_{n,M+1}-\vf_{M+1} \\
 \vdots &&\vdots\\
 \Delta_\omega\Re \sum\limits_{n=0}^{N}\vF_{n,1}-\vf_{1}&\ldots&  \Delta_\omega\Re \sum\limits_{n=0}^{N}\vF_{n,M+1}-\vf_{M+1} 
\end{pmatrix}\in \RR^{(M+1)\times (M+1)}.
\end{equation}
and the second term 
\begin{align}
	\label{eq:grad2}
	2\Delta_t\Delta_\omega\mu\mathcal{B}^*\mathcal{B}\vF & = 2\Delta_t\Delta_\omega\mu\left( -\partial_t\partial_t\vF+i4\pi\boldsymbol{\omega}\partial_t\vF-\partial_t\partial_\omega\vF\text{diag}(\valpha)+ 4\pi^2\boldsymbol{\omega}^2\vF\right. 
	\\
	&\quad \left.+i 2\pi\boldsymbol{\omega}\partial_\omega\vF\text{diag}(\valpha) - \partial_\omega\partial_t\vF\text{diag}(\valpha)+i2\pi\partial_\omega \boldsymbol{\omega}\vF\text{diag}(\valpha) - \partial_\omega\partial_\omega\vF\text{diag}(\valpha)\right) .
	\nonumber
\end{align}

Similarly, by taking $\vbeta\in\CC^{M+1}$, the gradient of $\mathcal{G}_{\vF}$ at $\valpha$ after discretization is evaluated by
\begin{align}
\nabla\mathcal{G}_{\vF}|_{\valpha}\vbeta&=\,\lim_{h\to 0}\frac{\mathcal{G}_{\vF}(\valpha+h\vbeta)-\mathcal{G}_{\vF}(\valpha)}{h}\\
&=\,\text{tr}((\partial_\omega\vF\text{diag}(\vbeta))^*(\partial_\omega\vF\text{diag}(\valpha))\nonumber\\
&=\,-\vbeta^*\text{tr}(\vF^*\partial_\omega\partial_\omega\vF\text{diag}(\valpha)).\nonumber
\end{align}
Thus, we have
\begin{align}
\nabla\mathcal{G}_{\vF}|_{\valpha}=-\text{tr}(\vF^*\partial_\omega\partial_\omega\vF\text{diag}(\valpha))\in \CC^{M+1},
\end{align}
where
\begin{align}
(\nabla\mathcal{G}_{\vF}|_{\valpha})_{m}=\alpha(t_m)\sum_{n=-N}^{N}[\partial_\omega F(t_m,\omega_n)]^2.
\end{align}

\subsection{Minimize the functional $\mathcal{H}(F,\alpha)$}

We now have all the results needed to propose an optimization algorithm to minimize the functional $\mathcal{H}(F,\alpha)$. The minimization of $\mathcal{H}(F,G)$ in (\ref{functional4}) is the same so we skip it.
The functional we would like to minimize depends on two terms, $F$ and $\alpha$. While the PALM algorithm studied in~\cite{bolte2014proximal} provides a simple procedure to minimize~\eqref{functional3}, this algorithm appeared to be too slow in practice for this problem. Since the functional spaces $F$ and $\alpha$ live are convex, we will therefore minimize the functional alternately by optimizing one of these two terms when the other one is fixed; that is,
\begin{align}\label{Optimization:Alternating}
\left\lbrace 
\begin{array}{l}
F_{k+1} = 	\arg\min\limits_{F} \mathcal{H}(F,\alpha_{k}) \\
\alpha_{k+1} = 	\arg\min\limits_{\alpha} \mathcal{H}(F_{k+1},\alpha).
\end{array}
\right. 
\end{align}
with $\alpha_0 = 0$ and $F_0=0$ are used to initialize the algorithm. A discussion on convergence results of this classical Gauss-Seidel method can be found in~\cite{bolte2014proximal}. 

As we will see in next subsections, if we can reach the global minimizer of $\alpha\mapsto\mathcal{H}(F_{k+1},\alpha)$, finding a minimizer of $F\mapsto\mathcal{H}(F,\alpha_{k})$ requires the use of an iterative algorithm. We provide in~\ref{appendix:convergence} a convergence result of the practical algorithm we propose.

\subsection{Minimization of $\mathcal{H}_{\alpha} := \mathcal{H}(\cdot,\alpha)$}

When $\alpha$ is fixed, $\mathcal{H}_{\alpha}$ is a convex non smooth functional, involving a convex and Lipschitz differentiable term (the function $\mathcal{G}_{\alpha}:=\mathcal{G}(\cdot,\alpha)$), and a convex but non-smooth term (the $\Psi_{\alpha}:= \Psi(\cdot,\alpha)$ regularizer). Popular proximal algorithms such as forward-bacward~\cite{combettes2005signal} or the Fast Iterative Shrinkage/Thresholding Algorithm (FISTA)~\cite{Beck_Teboulle:2009,chambolle2014convergence} can then be employed. FISTA has the great advantage to reach the optimal rate of convergence; that is, if $\check{\vF}$ is the convergence point, $\mathcal{H}_\alpha(\vF_{k}) - \mathcal{H}_\alpha(\check{\vF}) = \mathcal{O}\left(\frac{1}{k^2}\right)$, while the forward-backward procedure converge in $\mathcal{O}\left(\frac{1}{k}\right)$ (see~\cite{tseng2010approximation} for a great review of proximal methods and their acceleration). This speed of convergence is usually observed in practice~\cite{loris2009performance}, and has been confirmed in our experiments (not shown in this paper). Contrary to the forward-backward, one limitation of the original FISTA~\cite{Beck_Teboulle:2009} is that the convergence is proven only on the sequence $\left(\mathcal{H}_{\alpha}(\vF_k)\right)_{k}$ rather than on the iterates $(\vF_k)_k$. 
However, the latest study~\cite{chambolle2014convergence} gives a version of FISTA, which fills in this gap while maintaining the same convergence rate. As far as we know, it is the only algorithm with these two properties, and then will be use in the following. Yet another shortcoming of the original FISTA is that the algorithm does not produce a monotonic decreasing of the functional, but a monotonic version is available~\cite{beck2009fast} and is used in this paper.

In short, FISTA relies on three steps
\begin{enumerate}
	\item A gradient descent step on the smooth term $\mathcal{G}_{\valpha}$;
	\item A soft-shrinkage operation, known as the proximal step;
	\item A relaxation step.
\end{enumerate}
The algorithm is summarized in Algorithm~\ref{alg:FISTA}. In practice, the Lipschitz constant can be evaluated using a classical power iteration procedure, or using a backtracking step inside the algorithm (see~\cite{Beck_Teboulle:2009} for details). $\nabla\mathcal{G}_\alpha$ is given by Eq.~\eqref{eq:grad1} and Eq.~\eqref{eq:grad2}.

Moreover, when the signal $\vf$ is real and $\alpha$ is real, we can limit the optimization to the positive frequencies such that $\vF\in\CC^{(N+1)\times(M+1)}$, with $N=\lceil M/2\rceil$. Indeed, one can show that there exists a solution $\vF$ which has an Hermitian symmetry properties, i.e. such that $F(t,\omega) = \overline{F(t,-\omega)}$. In order to prove this result, we remark that we have
\begin{align}
	\overline{\nabla\mathcal{G}_\alpha|_{\vF}(t,-\omega)} = \nabla\mathcal{G}_\alpha|_{\overline{\vF}}\left(t,-\omega\right),
\end{align}
which can be easily checked thanks to Eq.~\eqref{eq:grad1} and~\eqref{eq:grad2}. Then, if $\vF_{0}$ is Hermitian symetric, one can prove by induction that at each iteration, $\vF_{k}$ is Hermitian symmetric.

\begin{algorithm}[h!]
\begin{algorithmic}
\STATE Choose a stopping value $\epsilon$.
\STATE The initial values are $\vF_{0}\in\CC^{(N+1)\times(M+1)}$, $z_{0}=\vF_0$ 
\STATE Evaluate the Lipschitz constant $L=\|\nabla\mathcal{G}_{\alpha}\|^2$ by power iterations.
\WHILE{$\frac{\|F_{k+1} - F_{k}\|}{\|F_k\|} > \epsilon$}
\STATE Gradient step: $\vF_{k+1/2} \gets z_{k} - \frac{1}{L}\nabla\mathcal{G}_\alpha|_{z_{k}}$ (see \eqref{eq:grad1} and \eqref{eq:grad2});
\STATE Proximal step: $\vF_{k+1/2} \gets \vF_{k+1/2}\left(1 - \frac{\lambda/L}{|F_{k+1/2}|} \right)^+$;
\STATE Monotonic step: 
\IF{$\mathcal{H}(F_{k+1/2},\alpha) < H(F_k,\alpha)$}
\STATE $F_{k+1} = F_{k+1/2}$
\ELSE
\STATE $F_{k+1} = F_k$
\ENDIF
\STATE Relaxation step: $z_{k+1}\gets \vF_{k+1} + \frac{k}{k+2} (\vF_{k+1}-\vF_{k}) + \frac{k+1}{k+2}(\vF_{k+1/2} - \vF_{k})$;
\STATE $k=k+1$;
\ENDWHILE
\STATE Output $\vF$.
\end{algorithmic}
\caption{FISTA algorithm for $\mathcal{H}_{\alpha}$: $\vF  = \text{FISTA}(\vF_{0},\alpha,\epsilon)$}
\label{alg:FISTA}
\end{algorithm}

\subsection{Minimization of $\mathcal{H}_{F} := \mathcal{H}(F,\cdot)$}

Once $\vF_{k}$ is estimated, the minimization of $\mathcal{H}_{\vF_{k}}$ reduces to a simple quadratic minimization: 
\begin{align}
	\valpha_{k+1} = \argmin_{\valpha} &\left\{\mu\sum\limits_{m=0}^{M}\sum\limits_{n=0}^{N}\left|\partial_{t}F(t_{m},\omega_{n})-i2\pi\omega_{n}F(t_{m},\omega_{n})+\alpha(t_m)\partial_{\omega}F(t_{m},\omega_{n})\right|^{2} \right.\nonumber\\
	&\qquad\qquad\left. + \gamma \sum_{m=0}^{M} |\alpha(t_m)|^2\right\}.
\end{align}
Thus, $\alpha$ can be estimated in a closed form as, for all $m=0,\ldots,M$,
\begin{equation}
	\label{eq:alpha_est}
	\alpha_{k+1}(t_m)
	=  \frac{2\sum\limits_{n=0}^N \Re\left(\partial_{\omega}\overline{F(t_{m},\omega_{n})}\big[\partial_{t}F(t_{m},\omega_{n})-i2\pi\omega_{n}F(t_{m},\omega_{n})\big]\right) }{\sum\limits_{n=0}^N|\partial_{\omega}F(t_{m},\omega_{n})|^2 + \gamma/\mu}.	
\end{equation}

\subsection{General algorithm}

We summarize in Algorithm~\ref{alg:alter} the practical procedure to minimize $\mathcal{H}$~\eqref{functional3}. The choices of the parameters are discussed below.

\begin{algorithm}[h!]
\begin{algorithmic}
\STATE Choose a stopping value $\epsilon_1$ for the FISTA algorithm;
\STATE Choose a stopping value $\epsilon_2$ for the alternating minimization;
\STATE Choose a set of decreasing values $I_{\tilde\mu}$ for the parameter $\tilde\mu\in\RR_+$. 
\STATE Choose the parameters $\tilde \lambda\in [0,1]$ and $\gamma \in \RR_+$;
\STATE The initial values are $k=0$, $\vF_{0} = \boldsymbol{0}$, $\alpha_{0} = \boldsymbol{0}$;
\FOR{$\tilde \mu\in I_{\tilde\mu}$}
\WHILE{$\frac{\|F_{k+1} - F_{k}\|}{\|F_k\|} > \epsilon_1$}
\STATE FISTA step: $\vF_{k+1} = \text{FISTA}(F_{k},\alpha_k,\epsilon_1)$ (see Alg.~\ref{alg:FISTA});
\STATE alpha estimation step (see Eq.~\eqref{eq:alpha_est});
\STATE $k=k+1$;
\ENDWHILE
\ENDFOR
\STATE Output $\vF,\alpha$;
\end{algorithmic}
\caption{Algorithm for minimization of $\mathcal{H}$}
\label{alg:alter}
\end{algorithm}

\begin{itemize}
	\item {\bf Stopping criterion.} As the functional $F\mapsto \mathcal{H(F,\alpha)}$ is convex, a good stopping criterion for FISTA is the so-called {\em duality gap}. However, the duality gap cannot be computed easily here. We then choose the classical quantity, $\frac{\|F_{k+1}-F_{k}\|}{\|F_k\|}$, to stop the FISTA inner loop as well as the alternating algorithm; that is, the algorithm stops when both the stopping criteria, $\frac{\|F_{k+1}-F_{k}\|}{\|F_k\|}\leq \epsilon_1$ and $\frac{\|\alpha_{k+1}-\alpha_{k}\|}{\|\alpha_k\|}\leq \epsilon_2$ for the chosen $\epsilon_1,\epsilon_2\geq 0$, are satisfied. $\epsilon_1$ and $\epsilon_2$ can be set to $5\times 10^{-4}$ in practice: smaller value produce a much slower algorithm for similar results.
	\item {\bf Set of values $I_{\tilde\mu}$}. A practical choice of $I_{\tilde\mu}$ is a set of $K$ values uniformly distributed on the logarithmic scale. In the noise free case, one must choose a sufficiently small $\tilde \mu$. However, a small value of $\tilde \mu$ gives a very slow algorithm. A practical strategy is to use a fixed point continuation~\cite{hale2008fixed}, also known as warm start, strategy to minimize $\mathcal{H}$. If the noise is taken into account, the final $\tilde \mu$ cannot be known in advance, but can be chosen to be the one leading to the best result among the $K$ obtained minimizers. Here, we choose $\tilde \mu$ according to the discrepancy principle~\cite{morozov1966solution}. Another approach could be the GSURE approach~\cite{deledalle2012proximal} (not derived in this work).
	\item {\bf Parameter $\tilde \lambda$.} This parameter must be chosen between $0$ and $1$. The closer $\tilde \lambda$ is to $1$, the more importance is given to the constraints on the derivatives. As these constraints should be satisfied as much as possible, we choose in practice $\tilde \lambda \simeq 0.99$. 
	\item {\bf Parameter $\gamma$.} The influence of this parameter is not dominant on the results. We set $\gamma\simeq 10^{-3}$ in order to prevent any division by $0$ during the estimation of $\alpha$ by \eqref{eq:alpha_est}.
	\item {\bf Initialization of the algorithm.} The choice of $\alpha = \boldsymbol{0}$ appears to be natural, as we cannot have access to the chirp factor. The first iteration of Algorithm \ref{alg:alter} is equivalent to an estimation without taking this chirp factor into account. 
However, this initialization can have some influence on the speed of the algorithm~\cite{Beck_Teboulle:2009}. As the solution is expected to be sparse, $\vF=\boldsymbol{0}$ seems to be a reasonable choice.
\end{itemize}

\section{Numerical Results}\label{Section:NumericalResults}

In this section we show numerical simulation results of the proposed algorithm. The code and simulated data are available via request. In this section, we take $W$ to be the standard Brownian motion defined on $[0,\infty)$ and define a {\it smoothed Brownian motion with bandwidth $\sigma>0$} as
\begin{align}
\Phi_{\sigma}:=W\star K_{\sigma},
\end{align}
where $K_{\sigma}$ is the Gaussian function with the standard deviation $\sigma>0$ and $\star$ denotes the convolution operator. 

\subsection{Single component, noise-free}
The first example is a semi-real example which is inspired from a medical challenge. Atrial fibrillation (Af) is a pathological condition associated with high mortality and morbidity \cite{Jahangir_Lee_Friedman_Trusty_Hodge:2007}. It is well known that the subject with Af would have irregularly irregular heart beats. In the language under our framework, the instantaneous frequency of the electrocardiogram signal recorded from an Af patient varies fast. To study this kind of signal with fast varying instantaneous frequency, we pick a patient with Af and determine its instantaneous heart rate by evaluating its R peak to R peak intervals. Precisely, if the R peaks are located on $t_i$, we generate a non-uniform sampling of the instantaneous heart rate and denote it as $(t_i,1/(t_{i+1}-t_i))$. Then the instantaneous heart rate, denoted as $\phi_1'(t)$, is approximated by the cubic spline interpolation. 
Next, define another a random process $A_1$ on $[0,L]$ by
\begin{align}\label{Simulation:A_1A_2}
A_1(t)=1+ \frac{\Phi_{\sigma_1}(t) + \|\Phi_{\sigma_1}\|_{L^\infty[0,L]}}{2\|\Phi_{\sigma_1}\|_{L^\infty[0,L]}}, 
\end{align}
where $t\in[0,L]$ and $\sigma_1>0$.
Note that $A_1$ is a positive random process and in general there is no close form expression of $A_1(t)$ and $\phi_1(t)$. The dynamic of both components can be visually seen from the signal. 
We then generate an oscillatory signal with fast varying instantaneous frequency 
\begin{align}
f_1(t)=A_1(t)\cos(2\pi\phi_1(t)),
\end{align}
where $A_1(t)$ is a realization of the random process defined in (\ref{Simulation:A_1A_2}). We take $L=80$, sample $f_1$ with the sampling rate $\Delta t=1/10$, $\sigma_1=100$, $\sigma_2=200$. To compare the result with other methods, in addition to showing the result of the proposed algorithm, we also show the analysis results of STFT and synchrosqueezed STFT. In the STFT and synchrosqueezed STFT, we take the window function $g$ as a Gaussian function with the standard deviation $\sigma=1$. 
See Figure \ref{figE2} for the result. We mention that in this section, when we plot the tvPS, we compress its dynamical range by the following procedure. Denoted the discretized tvPS as ${\mathfrak{R}}\in\RR^{m\times n}$, where $m,n\in\NN$ stand for the number of discrete frequencies and the number of time samples, respectively. Set $M$ to be the $99.9\%$ quantile of the absolute values of all entries of ${\mathfrak{R}}$, then normalize the discretized tvPS by $M$, and obtain ${\widetilde{\mathfrak{R}}}\in\RR^{m\times n}$ so that ${\widetilde{\mathfrak{R}}}(i,j):=\text{max}\{M,{\mathfrak{R}}(i,j)\}$ for $i=1,\ldots,m$ and $j=1,\ldots,n$. Then plot a gray-scale visualization of ${\widetilde{\mathfrak{R}}}$ in the linear scale. 
From the figure, we see that the proposed algorithm Tycoon could extract this kind of fast varying IF well visually. However, although there are some periods where STFT and synchrosqueezed STFT show a dominant curve following the IF well, in general the IF information is blurred in their TF representations. In addition, by Tycoon, the chirp factor can be approximated up to some extent.

\begin{figure}[h!]
    \begin{center}
    \includegraphics[width=\textwidth]{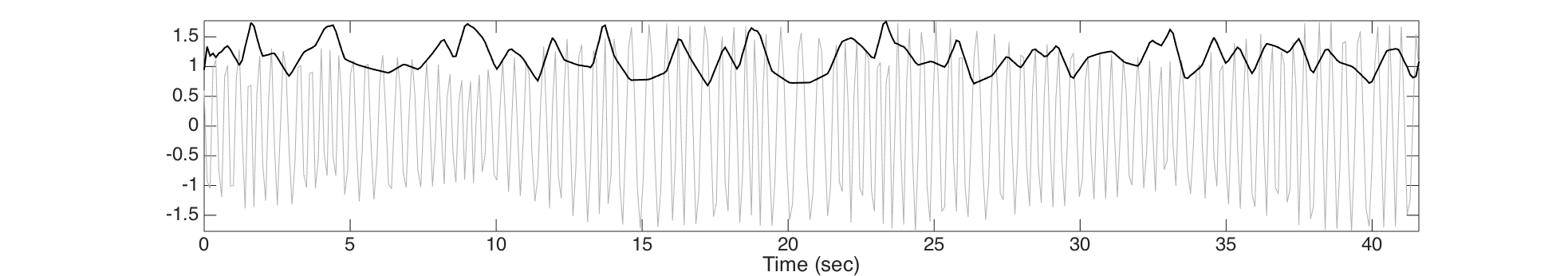}      \\
    \includegraphics[width=.45\textwidth]{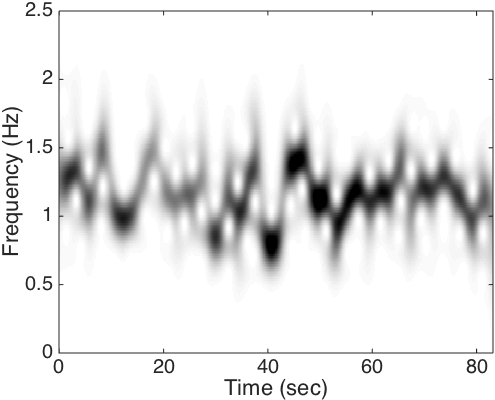}       
    \includegraphics[width=.45\textwidth]{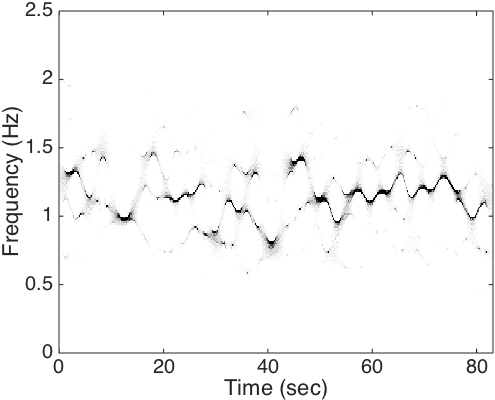}     \\
    \includegraphics[width=.45\textwidth]{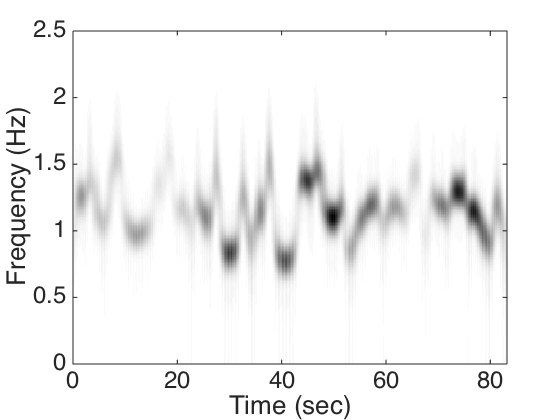}     
    \includegraphics[width=.45\textwidth]{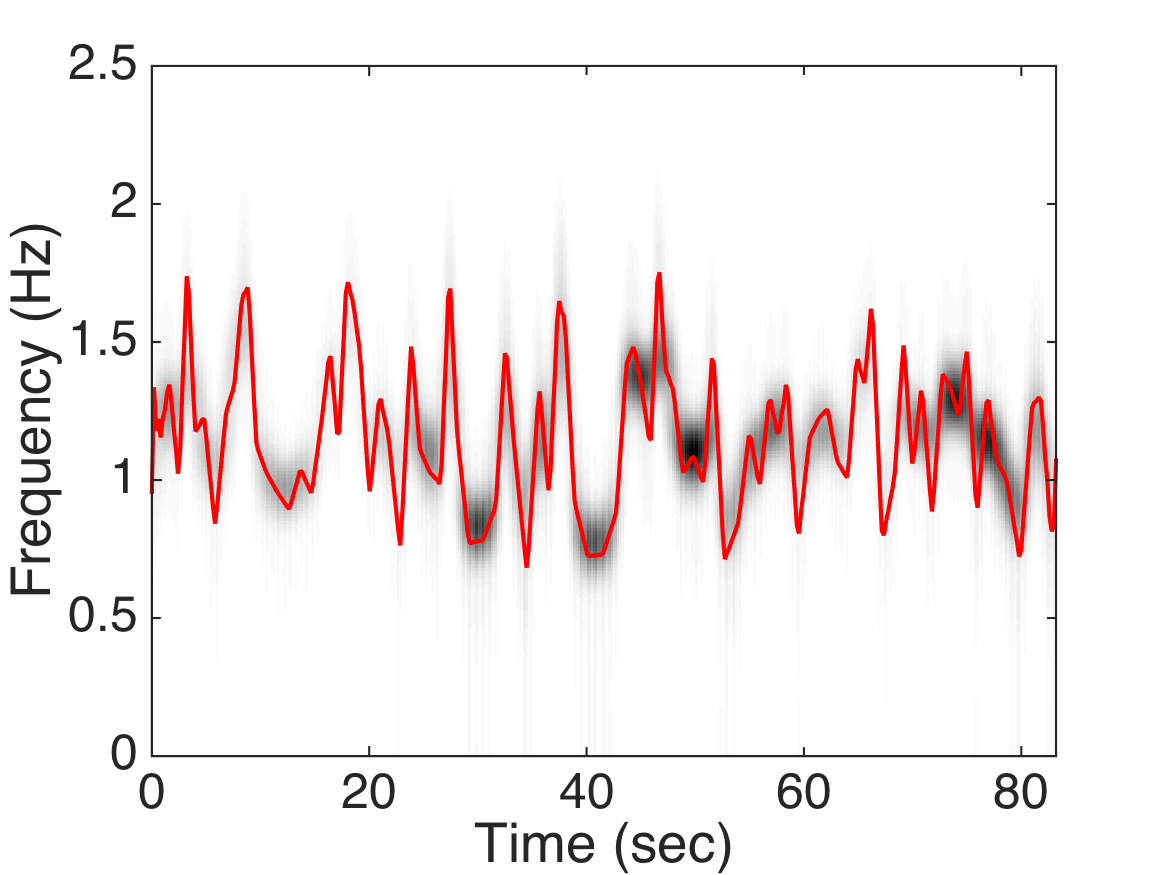}       \\
    \includegraphics[width=\textwidth]{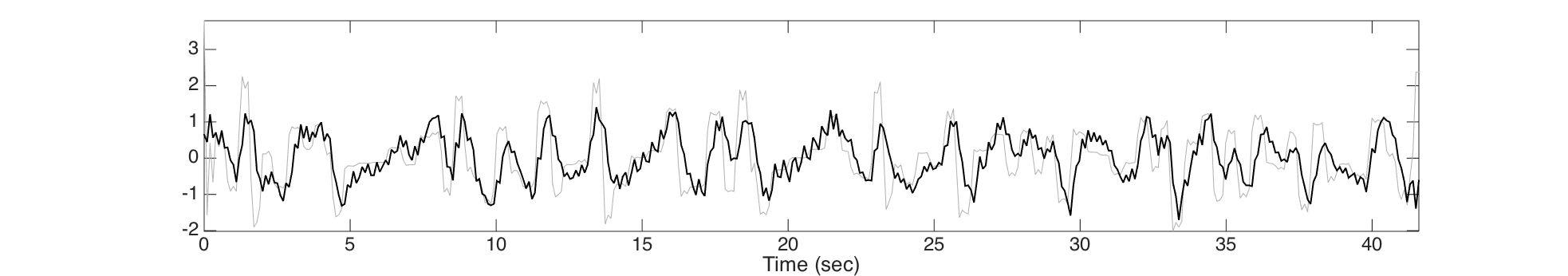}       
    \end{center}
    \caption{Top: the signal $f_1$ is shown as the gray curve with the instantaneous frequency superimposed as the black curve. It is clear that the instantaneous frequency varies fast. In the second row, the short time Fourier transform with the Gaussian window with the standard deviation $1$ is shown on the left and the synchrosqueezed short time Fourier transform is shown on the right. In the third row, the Tycoon result is shown on the left and our result with the instantaneous frequency superimposed as a red curve is shown on the right. In the bottom, the chirp factor, $\phi_2''(t)$, is shown as the gray curve and the estimated $\phi_2''(t)$; that is, the $\alpha(t)$, is properly normalized and superimposed as the black curve. In the top and bottom figures, for the sake of visibility, only the first part of the signal is demonstrated.}\label{figE2}
  \end{figure}

\subsection{Two components, noise-free}

In the second example, we consider an oscillatory signal with two gIMTs. Define random processes $A_2(t)$ and $\phi_2(t)$ on $[0,L]$ by
\begin{align}
A_2(t)&=1+\frac{\Phi_{\sigma_1}(t) + 2\|\Phi_{\sigma_1}\|_{L^\infty[0,L]}}{3\|\Phi_{\sigma_1}\|_{L^\infty[0,L]}},\\
\phi_2(t)&=\pi t+\int_{0}^t\left[\frac{\Phi_{\sigma_2}(s) + 0.5\|\Phi_{\sigma_2}\|_{L^\infty[0,L]}}{1.5\|\Phi_{\sigma_2}\|_{L^\infty[0,L]}}-\sin(s)\right]\ud s,\nonumber
\end{align}
where $t\in[0,L]$ and $\sigma_2>0$. Note that by definition $\phi_2$ are both monotonically increasing random processes. The signal is constructed as 
\begin{equation}\label{Numerical:Examplef}
f(t)=f_1(t)+f_2(t),
\end{equation}
where $f_2(t)=A_2(t)\cos(2\pi\phi_2(t))\chi_{[20,80]}(t)$ and $\chi$ is the indicator function. Again, we take $\sigma_1=100$, $\sigma_2=200$, $L=80$ and sample $f$ with the sampling rate $\Delta t=1/10$. The result is shown in Figure \ref{figE3}.
For the comparison purpose, we also show results from other TF analysis methods. In STFT and synchrosqueezed STFT, the window function is the same as that in the first example -- the Gaussian window with the standard deviation $\sigma=1$. We also show the result with the synchrosqueezed CWT \cite{Daubechies_Lu_Wu:2011,Chen_Cheng_Wu:2014}, where the mother wavelet $\psi$ is chosen to satisfy $\hat{\psi}(\xi)=e^{\frac{1}{(\frac{\xi-1}{0.2})^2-1}}\chi_{[0.8,1.2]}$, where $\chi$ is the indicator function. Further, the popular empirical mode decomposition algorithm combined with the Hilbert spectrum (EMD-HS) \cite{Huang_Shen_Long_Wu_Shih_Zheng_Yen_Tung_Liu:1998} is also evaluated. The tvPS of $f$ determined by EMD-HS is via the following steps. First, we run the proposed sifting process and decompose the given signal $f$ into $K$ components and the remainder term (see \cite{Huang_Shen_Long_Wu_Shih_Zheng_Yen_Tung_Liu:1998} for details of the sifting process); that is, $f(t)=\sum_{k=1}^{K_{\mathfrak{H}}}x_k(t) +r(t)$, where $K_{\mathfrak{H}}\in\NN$ is chosen by the user, $x_k$ is the $k$-th decomposed oscillatory component and $r$ is the remainder term. The IF and AM of the $k$-th oscillatory component is determined by the Hilbert transform; that is, by $\tilde{x}_k(t)= x_k(t) + i\mathcal{H}(x_k(t)) = b_k(t)e^{i2\pi\psi_k(t)}$, where $\mathcal{H}$ is the Hilbert transform, the IF and the AM of the $k$-th oscillatory component are estimated by $\psi_k'(t)$ and $b_k(t)$. Here we assume that $x_k$ is well-behaved so that the Hilbert transform works. Finally, the tvPS (or called the Hilbert spectrum in the literature) of the signal $f$ determined by the EMD-HS, denoted as $\mathfrak{H}_f$, is set to be $\mathfrak{H}_f(t,\omega)= \sum_{k=1}^{K_{\mathfrak{H}}}b_k(t)\delta(\omega-\psi_k'(t)(t))$. In this work, due to the well-known mode-mixing issue of EMD and the number of components is not known a priori, we choose $K_{\mathfrak{H}}=6$ so that we could hope to capture all needed information. We mention that one possible approach to evaluate the IF and AM after the sifting process is applying the SST directly to $x_k(t)$; this combination has been shown useful in the strong field atomic physics \cite{Sheu_Wu_Hsu:2015}.
The results of STFT, synchrosqueezed STFT, synchrosqueezing CWT and EMD-HS are shown in Figure \ref{figE3other}. Visually, it is clear that the proposed convex optimization approach, Tycoon, provides the dynamical information hidden inside the signal $f$, since the IFs of both components are better extracted in Tycoon, while several visually obvious artifacts could not be ignored in other TF analyses. For example, although we could see the overall pattern of the IF of $f_2$ in the STFT, the interfering pattern could not be ignored. While the IF of $f_2$ could be well captured in synchrosqueezed CWT, the IF of $f_1$ is blurred; on the other hand, while the IF of $f_1$ could be well captured in EMD-HS, the IF of $f_2$ is blurred. Clearly, the IF patterns of both components could not be easily identified in the synchrosqueezed STFT.

\subsection{Performance quantification}
To further quantify the performance of Tycoon, we consider the following metric. As indicated above, we would expect to recover the itvPS. Thus, to evaluate the performance of Tycoon and have a comparison with other TF analyses, we would compare the time varying power spectrum (tvPS) determined by different TF analyses with the itvPS of the clean simulated signal $s$. If we view both the itvPS and the tvPS as distributions on the TF-plane, we could apply the {\em Optimal Transport} (OT) distance, which is also well known as the Earth Mover distance (EMD), to evaluate how different the obtained tvPS is from the itvPS \cite{Daubechies_Wang_Wu:2015}. We would refer the reader to \cite[section 2.2]{Villani:book} for its detail theory. Here we quickly summarize how it works. 
Given two probability measures on the same set, the OT-distance evaluate the amount of ``work'' needed to ``deform'' one into the other. 
Precisely, the OT-distance between two probability distributions $\mu$ and $\nu$ on a metric space $({\tt S},d)$ involves an optimization over all possible probability measures on ${\tt S}\times{\tt S}$ that have $\mu$ and $\nu$ as marginals, denoted as $\mathcal{P}(\mu,\nu)$, by
\begin{align}
d_{\mbox{\footnotesize{OT}}}(\mu,\nu):= \inf_{\rho \in \mathcal{P}(\mu,\nu)}\int\,d(x,y)\,
\ud \rho(x,y)~,
\end{align}
which in the one-dimensional case, that is, when ${\tt S}\subset \RR$, and $d$ is the canonical Euclidean distance, $d(x,y)=|x-y|$, could be easily evaluated. Define $f_\mu(x)=\int_{-\infty}^x\, \ud \mu $ and $f_\nu(x)=\int_{-\infty}^x\, \ud \nu $, the OT distance is reduced to the $L^1$ difference of $f_\mu$ and $f_\nu$; that is,
\begin{align}
d_{\mbox{\footnotesize{OT}}}(\mu,\nu)= \int_{\mbox{\tt S}}\, |f_\mu(x)-f_\nu(x)|\, \ud x ~.
\end{align}
In the TF representation, as tvPS is always non-negative, we could view the distribution of the tvPS at each time as a probability density after normalizing its $L^1$ to $1$. This distribution indicates how accurate the TF analyses recover the oscillatory behavior of the signal at each time. Thus, based on the OT distance, we consider the following $\mathfrak{D}$ metric to evaluate the performance of each TF analyses of the function $f$ by
\begin{align}
\mathfrak{D}:=100\times \int_{-\infty}^\infty d_{\mbox{\footnotesize{OT}}}(P_f^t,\tilde{P}_f^t)\,\,\ud t,
\end{align}
where $P_f^t(\omega):=\frac{S_f(t,\omega)}{\int_0^\infty S_f(t,\eta)\ud \eta}$, $\tilde{P}_f^t(\omega):=\frac{\tilde{S}_f(t,\omega)}{\int_0^\infty \tilde{S}_f(t,\eta)\ud \eta}$, $S_f(t,\omega)$ is the itvPS and $\tilde{S}_f(t,\omega)$ is the estimated tvPS by a chosen TF analysis. Clearly, the small the $\mathfrak{D}$ metric is, the better the itvPS is approximated. 


To evaluate the second example, we run STFT, synchrosqueezed STFT, synchrosqueezing CWT and Tycoon on $100$ different realizations of $f_2$ in (\ref{Numerical:Examplef}), and evaluate the $\mathfrak{D}$ metric. The result is displayed in (mean $\pm$ standard deviation). The $\mathfrak{D}$ metric between the itvPS and the tvPS determined by Tycoon (respectively, EMD-HS, STFT, synchrosqueezed STFT and synchrosqeezed CWT) is $6.06\pm 0.25$ (respectively, $7.18\pm 0.93$, $8.76\pm 0.41$, $8.13\pm 0.42$ and $7.36\pm 0.67$). Further, under the null hypothesis that there is no performance difference between the tvPS determined by Tycoon and STFT evaluated by the $\mathfrak{D}$ metric and we set the significant level at $5\%$, the t-test rejects the null hypothesis with the p-value less than $10^{-8}$. The same hypothesis testing results hold for the comparison between Tycoon and other methods. 
Note that while the performance of Tycoon seems better than EMD-HS, the $\mathfrak{D}$ metric only reflects partial information regarding the difference and more details should be taken into account to achieve a fair comparison. For example, if we set $K_{\mathfrak{H}}=2$, the $\mathfrak{D}$ metric between the itvPS and the tvPS determined by EMD-HS becomes $4.98\pm 0.81$, which might suggest that EMD-HS performs better. However, this ``better performance'' is not surprising since the sparsity property is perfectly satisfies in EMD-HS, which is inherited in the procedure, while the mode mixing issue might lead to wrong interpretation eventually. Note that it is also possible to post-process the outcome of the sifting process to enhance the result, but these ad-hoc post-processing again are not mathematically well supported. 
Since it is out of the scope of this paper, we would leave this overall comparison between different TF analyses based on different philosophy as well as a better metric to the future work. 

\begin{figure}[h!]
    \begin{center}
    \includegraphics[width=.9\textwidth]{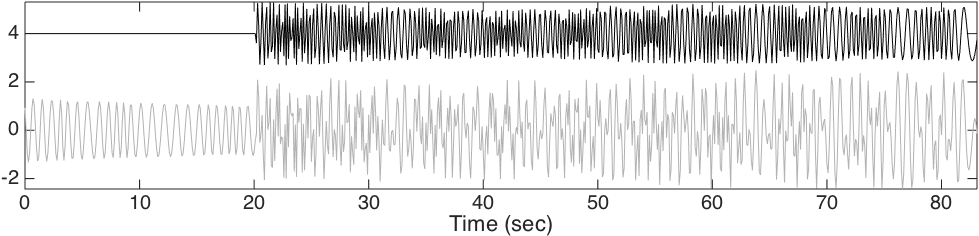}      \\
    \includegraphics[width=.45\textwidth]{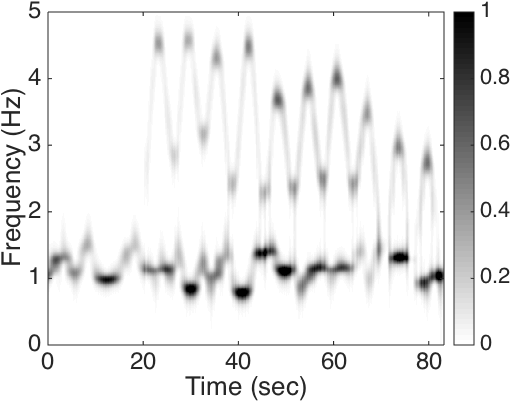}     
    \includegraphics[width=.45\textwidth]{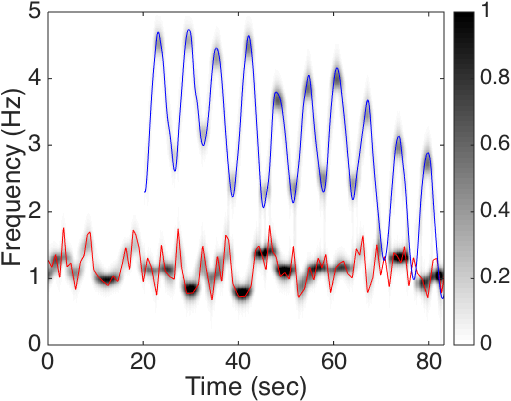}      
\end{center}
\caption{Top: the signal $f$ is shown as the gray curve with $f_2$ superimposed as the black curve which is shifted up by $4$ to increase the visualization. It is clear that the instantaneous frequency (IF) also varies fast in both components. 
In the bottom row, the intensity of the time frequency representation, $|\tilde{R}_f|^2$, determined by the proposed Tycoon algorithm is shown on the left; on the right hand side, the instantaneous frequencies associated with the two components are superimposed on $|\tilde{R}_f|^2$ as a red curve and a blue curve.}\label{figE3}
\end{figure}

\begin{figure}[h!]
    \begin{center}
    \includegraphics[width=.45\textwidth]{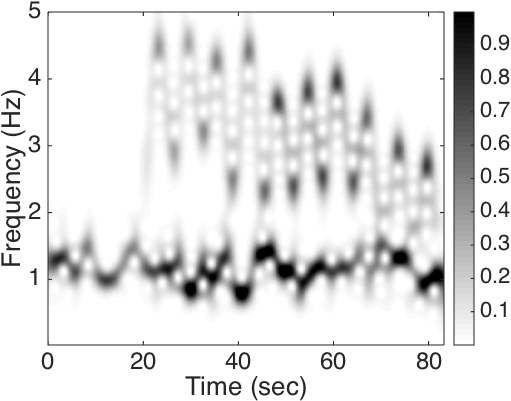}       
    \includegraphics[width=.45\textwidth]{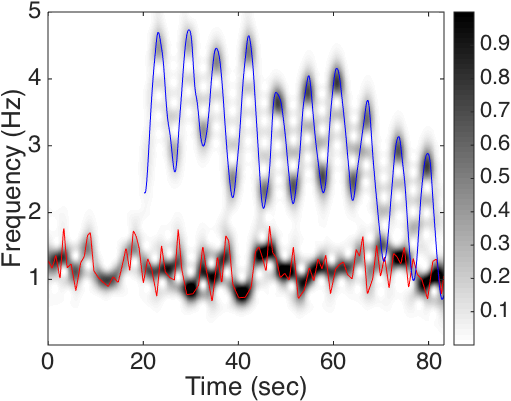}   \\
    \includegraphics[width=.45\textwidth]{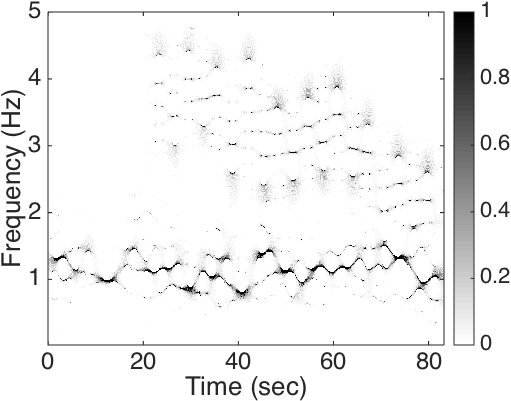}     
    \includegraphics[width=.45\textwidth]{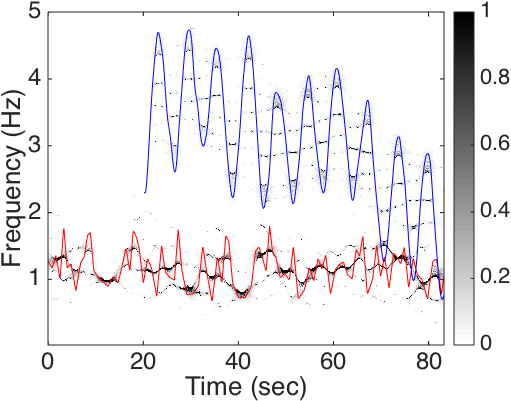}  \\   
    \includegraphics[width=.45\textwidth]{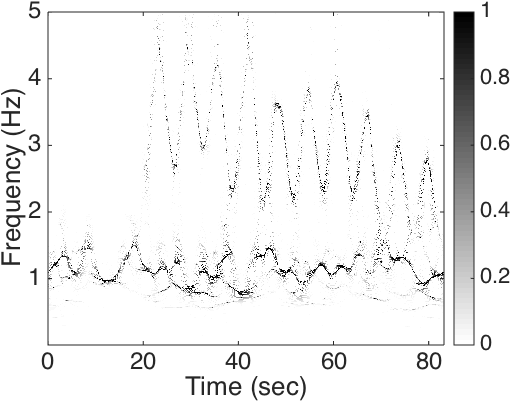}     
    \includegraphics[width=.45\textwidth]{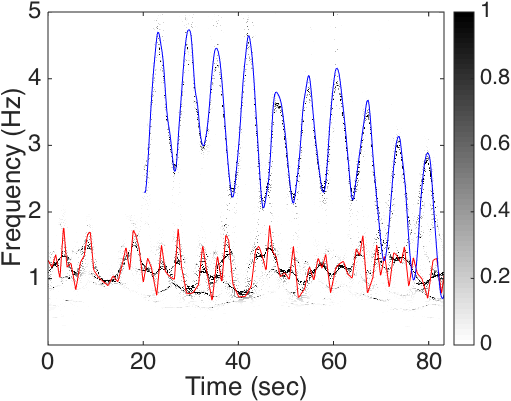}  \\
        \includegraphics[width=.45\textwidth]{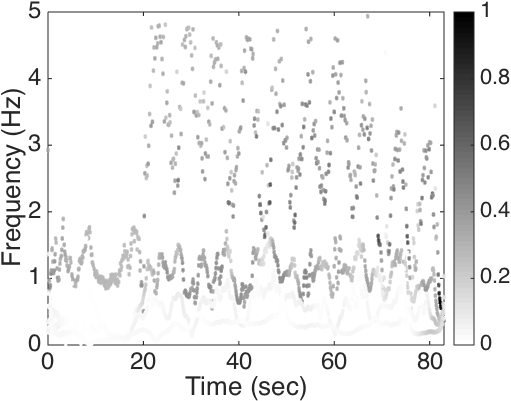}     
    \includegraphics[width=.45\textwidth]{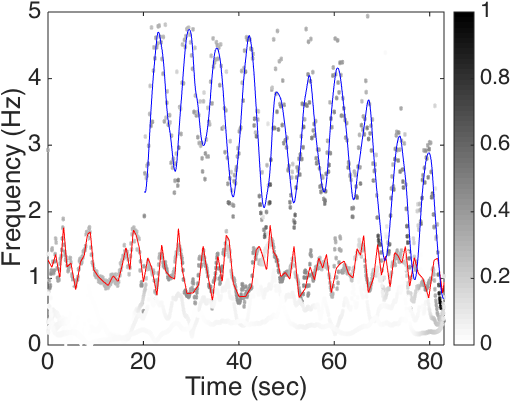}  
    \end{center}
    \caption{The time frequency (TF) representations of different TF analyses on the signal $f$. In the first row, on the left, the short time Fourier transform (STFT) with a Gaussian window with the standard deviation $\sigma=1$ is shown, and on the right the IF's of both components are superimposed for the visual comparison.
In the second row, on the left, the synchrosqueezed STFT with a Gaussian window with the standard deviation $\sigma=1$ is shown, and on the right the IF's of both components are superimposed for the visual comparison.
In the third row, on the left, we show the synchrosqueezed continuous wavelet transform with the mother wavelet $\psi$ so that $\hat{\psi}(\xi)=e^{\frac{1}{(\frac{\xi-1}{0.2})^2-1}}\chi_{[0.8,1.2]}$, where $\chi$ is the indicator function, and on the right the IF's of both components are superimposed for the visual comparison. It is clear that the slowly oscillatory component is not well captured.
In the bottom row, on the left, we show the TF representation determined by the empirical mode decomposition with the Hilbert transform, and on the right the IF's of both components are superimposed for the visual inspection. It is clear that the fast oscillatory component is not well captured.
}\label{figE3other}
  \end{figure}

\subsection{Two component, noisy}
In the third example, we add noise to the signal $f$ and see how the proposed algorithm performs. To model the noise, we define the signal to noise ratio (SNR) as
\begin{equation}
\mbox{SNR}:=20\log_{10}\frac{\mbox{std}(f)}{\mbox{std}(\Phi)},
\end{equation}
where $f$ is the clean signal, $\Phi$ is the added noise and $\mbox{std}$ means the standard deviation. In this simulation, we add the Gaussian white noise with SNR $7.25$ to the clean signal $f$, and obtain a noisy signal $Y$. The result is shown in Figure \ref{figE3noisy}. Clearly, we see that even when noise exists, the algorithm provides a reasonable result.
To further evaluate the performance, we run STFT, synchrosqueezed STFT, synchrosqueezing CWT and Tycoon on $100$ different realizations of $f_2$ in (\ref{Numerical:Examplef}) as well as $100$ different realizations of noise, and evaluate the $\mathfrak{D}$ metric. Here we use the same parameters as those in the second example to run STFT, synchrosqueezed STFT  and synchrosqueezed CWT. Since it is well known that EMD is not robust to noise, we replace the sifting process in EMD by that of the ensemble EMD (EEMD) to decompose the signal into $K_{\mathfrak{H}}=6$ oscillatory components, and generate the tvPS by the Hilbert transform as that in EMD. We call the method EEMD-HS. See \cite{Wu_Huang:2009} for the detail of the EEMD algorithm. The $\mathfrak{D}$ metric between the itvPS and the tvPS determined by Tycoon (respectively, EEMD-HS, STFT, synchrosqueezed STFT and synchrosqeezed CWT) is $11.87\pm 0.74$ (respectively, $11.65\pm 0.63$, $14.53\pm 0.55$, $14.09\pm 0.58$ and $12.79\pm 0.69$). The same hypothesis testing shows the significant difference between the performance of Tycoon and that of STFT, synchrosqueezed STFT and synchrosqeezed CWT, while there is no significant difference between the performance of Tycoon and that of EEMD-HS. Again, the same comments for the comparison between Tycoon and EMD-HS carry here when we compare Tycoon and EEMD-HS, and we leave the details to the future work.

\begin{figure}[h!]
    \begin{center}
    \includegraphics[width=.9\textwidth]{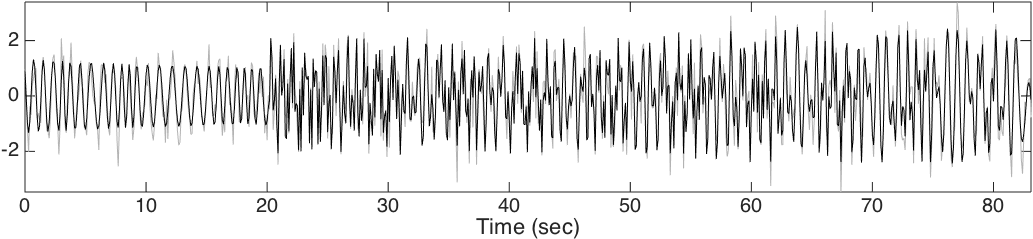}      \\
    \includegraphics[width=.45\textwidth]{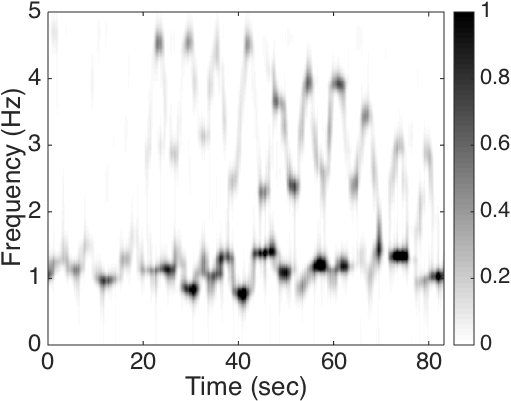}     
    \includegraphics[width=.45\textwidth]{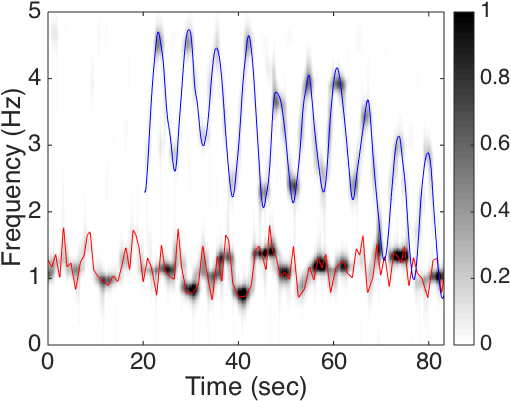}      
    \end{center}
    \caption{Top: the noisy signal $Y$ is shown as the gray curve with the clean signal $f$ superimposed as the black curve.In the second row, the intensity of the time frequency representation, $|\tilde{R}_Y|^2$, determined by our proposed Tycoon algorithm is shown on the left; on the right hand side, the instantaneous frequencies associated with the two components are superimposed on $|\tilde{R}_Y|^2$ as a red curve and a blue curve.}\label{figE3noisy}
  \end{figure}

\section{Discussion and future work}
In this paper we propose a generalized intrinsic mode functions and adaptive harmonic model to model oscillatory functions with fast varying instantaneous frequency. A convex optimization approach to find the time-frequency representation, referred to as Tycoon algorithm, is proposed. While the numerical results are encouraging, there are several things we should discuss. 
\begin{enumerate}
\item While with the help of FISTA the optimization process can be carried out, it is still not numerically efficient enough for practical usage. For example, it takes about 3 minutes to finish analyzing a time series with $512$ points in the laptop, but in many problems the data length is of order $10^5$ or longer. Finding a more efficient strategy to carry out the optimization is an important future work. One possible solution is by the sliding window idea. For a given long time series $f$ of length $n$ and a length $m<n$, we could run the optimization consecutively on the subinterval $I_j:=[j-m,j+m]$ to determine the tvPS at time $j$. Thus, the overall computational complexity could be $O(F(m)n)$, where $F(m)$ is the complexity of running the optimization on the subinterval $I_j$.
\item When there are more than one oscillatory component, we could consider (\ref{functional4}) to improve the result. However, in practice it does not significantly improve the result. Since it is of its own interest, we decide to leave it to the future work.
\item While the Tycoon algorithm is not very much sensitive to the choice of parameters $\mu$, $\lambda$ and $\gamma$, how to choose an optimal set of parameters is left unanswered in the current paper.
\item The noise behavior and influence on the Tycoon algorithm is not clear at this moment, although we could see that it is robust to the existence of noise in the numerical section. Theoretically studying the noise influence on the algorithm is important for us to better understand what we see in practice. 
\end{enumerate}
Before closing the paper, we would like to indicate an interesting finding about SST which is related to our current study. When an oscillatory signal is composed of intrinsic mode type function with slowly varying IF, it has been studied that the time-frequency representation of a function depends ``weakly'' on a chosen window, when the window has a small support in the Fourier domain \cite{Daubechies_Lu_Wu:2011,Chen_Cheng_Wu:2014}. Precisely, the result depends only on the first three absolute moments of the chosen window and its derivative, but not depends on the profile of the window itself. However, the situation is different when we consider an oscillatory signal composed of gIMT function with fast varying IF. As we have shown in Figure \ref{figE3}, when the window is chosen to have a small support in the Fourier domain, the STFT and synchrosqueezed STFT results are not ideal. Nevertheless, nothing prevents us from trying a window with a small support in the time domain; that is, a wide support in the Fourier domain. As is shown in Figure \ref{figE4}, by taking the window to be a Gaussian function with the standard deviation $0.4$, STFT and synchrosqueezed STFT provide reasonable results for the signal $f$ considered in (\ref{Numerical:Examplef}). Note that while we could start to see the dynamics in both STFT and synchrosqueezed STFT, the overall performance is not as good as that provided by Tycoon. Since it is not the focus of the current paper, we just indicate the possibility of achieving a better time-frequency representation by choosing a suitable window in SST, but not make effort to determine the optimal window. This kind of approach has been applied to the strong field atomic physics \cite{Li_Sheu_Laughlin_Chu:2015,Sheu_Wu_Hsu:2015}, where the window is manually but carefully chosen to extract the physically meaningful dynamics. A theoretical study regarding this topic will be reported in the near future. 

\begin{figure}[h!]
    \begin{center}
    \includegraphics[width=.45\textwidth]{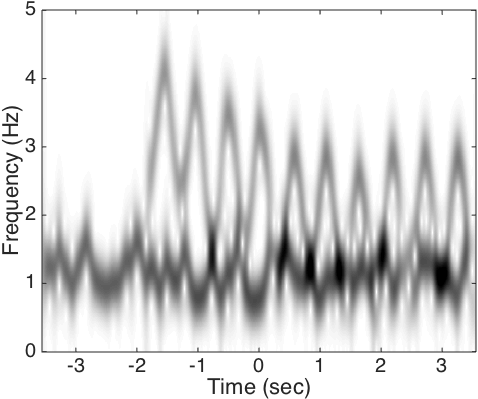}     
    \includegraphics[width=.45\textwidth]{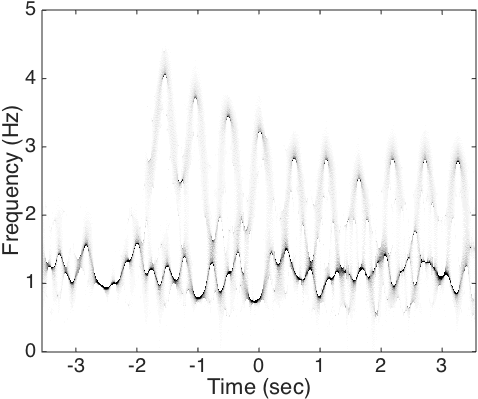}      
    \end{center}
    \caption{Left: the intensity of the short time Fourier transform (STFT) with a Gaussian window with the standard deviation $\sigma=0.4$ is shown on the left and the intensity of the synchrosqueezed STFT is shown on the right.}\label{figE4}
  \end{figure}

\section{Acknowledgement}
Hau-tieng Wu would like to thank Professor Ingrid Daubechies and Professor Andrey Feuerverger for their valuable discussion and Dr. Su Li for discussing the application direction in music and sound analyses. Hau-tieng Wu's work is partially supported by Sloan Research Fellow FR-2015-65363. Part of this work was done during Hau-tieng Wu's visit to National Center for Theoretical Sciences, Taiwan, and he would like to thank NCTS for its hospitality.
Matthieu Kowalski benefited from the support of the ``FMJH Program Gaspard Monge in optimization and operation research'', and from the support to this program from EDF.
We would also like to thank the anonymous reviewers for their constructive and helpful comments.

\bibliographystyle{plain}
\bibliography{TFanalysis,Optimization}

\begin{thebibliography}{10}

\bibitem{Auger_Chassande-Mottin_Flandrin:2012}
F.~Auger, E.~Chassande-Mottin, and P.~Flandrin.
\newblock Making reassignment adjustable: The levenberg-marquardt approach.
\newblock In {\em Acoustics, Speech and Signal Processing (ICASSP), 2012 IEEE
  International Conference on}, pages 3889--3892, March 2012.

\bibitem{Auger_Flandrin:1995}
F.~Auger and P.~Flandrin.
\newblock Improving the readability of time-frequency and time-scale
  representations by the reassignment method.
\newblock {\em IEEE Trans. Signal Process.}, 43(5):1068 --1089, may 1995.

\bibitem{Balazs_Dorfler_Jaillet_Holighaus_Velasco:2011}
P.~Balazs, M.~D{\"{o}}rfler, F.~Jaillet, N.~Holighaus, and G.~Velasco.
\newblock {Theory, implementation and applications of nonstationary Gabor
  frames}.
\newblock {\em Journal of Computational and Applied Mathematics},
  236(6):1481--1496, 2011.

\bibitem{beck2009fast}
A.~Beck and M.~Teboulle.
\newblock Fast gradient-based algorithms for constrained total variation image
  denoising and deblurring problems.
\newblock {\em Image Processing, IEEE Transactions on}, 18(11):2419--2434,
  2009.

\bibitem{Beck_Teboulle:2009}
A.~Beck and M.~Teboulle.
\newblock A fast iterative shrinkage-thresholding algorithm for linear inverse
  problems.
\newblock {\em SIAM J. Imaging Sciences}, 2(1):183--202, 2009.

\bibitem{bolte2014proximal}
J{\'e}r{\^o}me Bolte, Shoham Sabach, and Marc Teboulle.
\newblock Proximal alternating linearized minimization for nonconvex and
  nonsmooth problems.
\newblock {\em Mathematical Programming}, 146(1-2):459--494, 2014.

\bibitem{chambolle2014convergence}
A.~Chambolle and C.~Dossal.
\newblock On the convergence of the iterates of {FISTA}.
\newblock {\em Preprint hal-01060130, September}, 2014.

\bibitem{Chassande-Mottin_Auger_Flandrin:2003}
E.~Chassande-Mottin, F.~Auger, and P.~Flandrin.
\newblock Time-frequency/time-scale reassignment.
\newblock In {\em Wavelets and signal processing}, Appl. Numer. Harmon. Anal.,
  pages 233--267. Birkh\"auser Boston, Boston, MA, 2003.

\bibitem{Chen_Cheng_Wu:2014}
Y.-C. Chen, M.-Y. Cheng, and H.-T. Wu.
\newblock {Nonparametric and adaptive modeling of dynamic seasonality and trend
  with heteroscedastic and dependent errors}.
\newblock {\em J. Roy. Stat. Soc. B}, 76:651--682, 2014.

\bibitem{Chui_Lin_Wu:2014}
C.~K. Chui, Y.-T. Lin, and H.-T. Wu.
\newblock Real-time dynamics acquisition from irregular samples -- with
  application to anesthesia evaluation.
\newblock {\em Analysis and Applications, accepted for publication}, 2015.
\newblock DOI: 10.1142/S0219530515500165.

\bibitem{Chui_Mhaskar:2015}
C.~K. Chui and H.N. Mhaskar.
\newblock Signal decomposition and analysis via extraction of frequencies.
\newblock {\em Appl. Comput. Harmon. Anal.}, 2015.

\bibitem{Cicone_Liu_Zhou:2014}
A.~Cicone, J.~Liu, and H.~Zhou.
\newblock Adaptive local iterative filtering for signal decomposition and
  instantaneous frequency analysis.
\newblock {\em arXiv preprint arXiv:1411.6051}, 2014.

\bibitem{Cicone_Zhou:2015}
A.~Cicone and H.~Zhou.
\newblock Multidimensional iterative filtering method for the decomposition of
  high-dimensional non-stationary signals.
\newblock {\em arXiv preprint arXiv:1507.07173}, 2015.

\bibitem{combettes2005signal}
P.~L. Combettes and V.~R. Wajs.
\newblock Signal recovery by proximal forward-backward splitting.
\newblock {\em Multiscale Modeling \& Simulation}, 4(4):1168--1200, 2005.

\bibitem{Daubechies_Lu_Wu:2011}
I.~Daubechies, J.~Lu, and H.-T. Wu.
\newblock Synchrosqueezed wavelet transforms: An empirical mode
  decomposition-like tool.
\newblock {\em Appl. Comput. Harmon. Anal.}, 30:243--261, 2011.

\bibitem{Daubechies_Maes:1996}
I.~Daubechies and S.~Maes.
\newblock {A nonlinear squeezing of the continuous wavelet transform based on
  auditory nerve models}.
\newblock {\em Wavelets in Medicine and Biology}, pages 527--546, 1996.

\bibitem{Daubechies_Wang_Wu:2015}
I.~Daubechies, Y.~Wang, and H.-T. {Wu}.
\newblock {ConceFT: Concentration of frequency and time via a multitapered
  synchrosqueezing transform}.
\newblock {\em Philosophical Transactions A}, Accepted for publication, 2015.

\bibitem{DeLivera_Alysha_Hyndman_Snyder:2011}
A.~M. {De Livera}, R.~J. Hyndman, and R.~D. Snyder.
\newblock {Forecasting Time Series With Complex Seasonal Patterns Using
  Exponential Smoothing}.
\newblock {\em J. Am. Stat. Assoc.}, 106(496):1513--1527, 2011.

\bibitem{deledalle2012proximal}
C.~Deledalle, S.~Vaiter, G.~Peyr{\'e}, J.~Fadili, and C.~Dossal.
\newblock Proximal splitting derivatives for risk estimation.
\newblock {\em Journal of Physics: Conference Series}, 386(1):012003, 2012.

\bibitem{Dragomiretskiy_Zosso:2014}
K.~Dragomiretskiy and D.~Zosso.
\newblock {Variational Mode Decomposition}.
\newblock {\em IEEE Trans. Signal Process.}, 62(2):531--544, 2014.

\bibitem{Flandrin:1999}
P.~Flandrin.
\newblock {\em Time-frequency/time-scale analysis}, volume~10 of {\em Wavelet
  Analysis and its Applications}.
\newblock Academic Press Inc., 1999.

\bibitem{Flandrin:2001}
P.~Flandrin.
\newblock Time frequency and chirps.
\newblock In {\em Proc. SPIE}, volume 4391, pages 161--175, 2001.

\bibitem{Galiano_Velasco:2014}
G.~Galiano and J.~Velasco.
\newblock {On a non-local spectrogram for denoising one-dimensional signals}.
\newblock {\em Applied Mathematics and Computation}, 244:1--13, 2014.

\bibitem{Gilles:2013}
J.~Gilles.
\newblock {Empirical Wavelet Transform}.
\newblock {\em IEEE Trans. Signal Process.}, 61(16):3999--4010, 2013.

\bibitem{hale2008fixed}
E.~T. Hale, W.~Yin, and Y.~Zhang.
\newblock Fixed-point continuation for $\ell_1$-minimization: Methodology and
  convergence.
\newblock {\em SIAM Journal on Optimization}, 19(3):1107--1130, 2008.

\bibitem{Hou_Shi:2013a}
T.~Hou and Z.~Shi.
\newblock Data-driven time-frequency analysis.
\newblock {\em Appl. Comput. Harmon. Anal.}, 35(2):284 -- 308, 2013.

\bibitem{Hou_Shi:2013b}
T.~Hou and Z.~Shi.
\newblock Sparse time-frequency representation of nonlinear and nonstationary
  data.
\newblock {\em Science China Mathematics}, 56(12):2489--2506, 2013.

\bibitem{Hou_Shi:2011}
T.~Y. Hou and Z.~Shi.
\newblock Adaptive data analysis via sparse time-frequency representation.
\newblock {\em Adv. Adapt. Data Anal.}, 03(01n02):1--28, 2011.

\bibitem{Huang_Wang_Yang:2009}
C.~Huang, Y.~Wang, and L.~Yang.
\newblock Convergence of a convolution-filtering-based algorithm for empirical
  mode decomposition.
\newblock {\em Adv. Adapt. Data Anal.}, 1(4):561--571, 2009.

\bibitem{Huang_Shen_Long_Wu_Shih_Zheng_Yen_Tung_Liu:1998}
N.~E. Huang, Z.~Shen, S.~R. Long, M.C. Wu, H.H. Shih, Q.~Zheng, N.-C. Yen,
  C.~C. Tung, and H.~H. Liu.
\newblock {The empirical mode decomposition and the Hilbert spectrum for
  nonlinear and non-stationary time series analysis}.
\newblock {\em Proc. R. Soc. Lond. A}, 454(1971):903--995, 1998.

\bibitem{Huang_Zhang_Zhao_Sun:2015}
Z.~Huang, J.~Zhang, T.~Zhao, and Y.~Sun.
\newblock Synchrosqueezing s-transform and its application in seismic spectral
  decomposition.
\newblock {\em Geoscience and Remote Sensing, IEEE Transactions on},
  PP(99):1--9, 2015.

\bibitem{Jahangir_Lee_Friedman_Trusty_Hodge:2007}
A~Jahangir, Lee V., P.A. Friedman, J.M. Trusty, D.O. Hodge, and {et al.}
\newblock Long-term progression and outcomes with aging in patients with lone
  atrial fibrillation: a 30-year follow-up study.
\newblock {\em Circulation}, 115:3050--3056, 2007.

\bibitem{Kodera_Gendrin_Villedary:1978}
K.~Kodera, R.~Gendrin, and C.~Villedary.
\newblock Analysis of time-varying signals with small bt values.
\newblock {\em {IEEE} Trans. Acoust., Speech, Signal Processing}, 26(1):64 --
  76, feb 1978.

\bibitem{Li_Liang:2012}
C.~Li and M.~Liang.
\newblock A generalized synchrosqueezing transform for enhancing signal
  time-frequency representation.
\newblock {\em Signal Processing}, 92(9):2264 -- 2274, 2012.

\bibitem{Li_Sheu_Laughlin_Chu:2015}
P.-C. Li, Y.-L. Sheu, C.~Laughlin, and S.-I Chu.
\newblock Dynamical origin of near- and below-threshold harmonic generation of
  {Cs} in an intense mid-infrared laser field.
\newblock {\em Nature Communication}, 6, 2015.

\bibitem{Lin_Wang_Zhou:2009}
L.~Lin, Y.~Wang, and H.~Zhou.
\newblock Iterative filtering as an alternative for empirical mode
  decomposition.
\newblock {\em Adv. Adapt. Data Anal.}, 1(4):543--560, 2009.

\bibitem{Liu_Hou_Shi:2015}
C.~Liu, T.~Y. Hou, and Z.~Shi.
\newblock On the uniqueness of sparse time-frequency representation of
  multiscale data.
\newblock {\em Multiscale Model. and Simul.}, 13(3):790--811, 2015.

\bibitem{loris2009performance}
Ignace Loris.
\newblock On the performance of algorithms for the minimization of
  ℓ1-penalized functionals.
\newblock {\em Inverse Problems}, 25(3):035008, 2009.

\bibitem{Mann_Haykin:1995}
S.~Mann and S.~Haykin.
\newblock {The chirplet transform: physical considerations}.
\newblock {\em Signal Process. IEEE Trans.}, 43(11):2745--2761, 1995.

\bibitem{morozov1966solution}
V.~A. Morozov.
\newblock On the solution of functional equations by the method of
  regularization.
\newblock {\em Soviet Math. Dokl}, 7(1):414--417, 1966.

\bibitem{Oberlin_Meignen_Perrier:2012}
T.~Oberlin, S.~Meignen, and V.~Perrier.
\newblock {An alternative formulation for the Empirical Mode Decomposition}.
\newblock {\em IEEE Trans. Signal Process.}, 60(5):2236--2246, 2012.

\bibitem{Oberlin_Meignen_Perrier:2015}
T.~Oberlin, S.~Meignen, and V.~Perrier.
\newblock Second-order synchrosqueezing transform or invertible reassignment?
  towards ideal time-frequency representations.
\newblock {\em IEEE Trans. Signal Process.}, 63(5):1335--1344, March 2015.

\bibitem{Pustelnik_Borgnat_Flandrin:2014}
N.~Pustelnik, P.~Borgnat, and P.~Flandrin.
\newblock Empirical mode decomposition revisited by multicomponent non-smooth
  convex optimization.
\newblock {\em Signal Processing}, 102(0):313 -- 331, 2014.

\bibitem{Ricaud_Stempfel_Torresani:2014}
B.~Ricaud, G.~Stempfel, and B.~Torr{\'{e}}sani.
\newblock {An optimally concentrated Gabor transform for localized
  time-frequency components}.
\newblock {\em Adv Comput Math}, 40:683--702, 2014.

\bibitem{Sheu_Wu_Hsu:2015}
Y.-L. Sheu, H.-T. Wu, and L.-Y. Hsu.
\newblock Exploring laser-driven quantum phenomena from a time-frequency
  analysis perspective: A comprehensive study.
\newblock {\em Optics Express}, 23:30459--30482, 2015.

\bibitem{Stockwell_Mansinha_Lowe:1996}
R.G. Stockwell, L.~Mansinha, and R.P. Lowe.
\newblock {Localization of the complex spectrum: the S transform}.
\newblock {\em Signal Process. IEEE Trans.}, 44(4):998--1001, 1996.

\bibitem{Tavallali_Hou_Shi:2014}
P.~Tavallali, T.~Hou, and Z.~Shi.
\newblock Extraction of intrawave signals using the sparse time-frequency
  representation method.
\newblock {\em Multiscale Modeling \& Simulation}, 12(4):1458--1493, 2014.

\bibitem{tseng2010approximation}
P.~Tseng.
\newblock Approximation accuracy, gradient methods, and error bound for
  structured convex optimization.
\newblock {\em Mathematical Programming}, 125(2):263--295, 2010.

\bibitem{Villani:book}
C.~Villanic.
\newblock {\em Topics in Optimal Transportation}.
\newblock Graduate Studies in Mathematics, American Mathematical Society, 2003.

\bibitem{Wu:2013}
H.-T. Wu.
\newblock Instantaneous frequency and wave shape functions {(I)}.
\newblock {\em Appl. Comput. Harmon. Anal.}, 35:181--199, 2013.

\bibitem{Wu_Huang:2009}
Z.~Wu and N.~E. Huang.
\newblock Ensemble empirical mode decomposition: a noise-assisted data analysis
  method.
\newblock {\em Adv. Adapt. Data Anal.}, 1:1 -- 41, 2009.

\bibitem{Yang:2014}
H.~{Yang}.
\newblock {Synchrosqueezed Wave Packet Transforms and Diffeomorphism Based
  Spectral Analysis for 1D General Mode Decompositions}.
\newblock {\em Appl. Comput. Harmon. Anal.}, 39:33--66, 2014.

\bibitem{zangwill1969nonlinear}
W.~I. Zangwill.
\newblock {\em Nonlinear programming: a unified approach}, volume 196.
\newblock Prentice-Hall Englewood Cliffs, NJ, 1969.

\end{thebibliography}

\clearpage

\appendix

\renewcommand{\theclaim}{A.\arabic{claim}}

\section{Proof of Theorem \ref{theorem:identifiability:single}}
Suppose 
\begin{equation}\label{observation:identifiability:lemma:1}
g(t)=a(t)\cos\phi(t)=(a(t)+\alpha(t))\cos(\phi(t)+\beta(t))\in \mathcal{Q}^{c_1,c_2,c_3}_\epsilon.
\end{equation}
Clearly we know $\alpha\in C^1(\RR),~\beta\in C^3(\RR)$. By the definition of $\mathcal{Q}^{c_1,c_2,c_3}_\epsilon$, we have
\begin{align}
&\inf_{t\in\RR}a(t)>c_1,~\sup_{t\in\RR}a(t)<c_2, \label{proof:2.1:Aeps:acond1}\\
&\inf_{t\in\RR}\phi'(t)>c_1,~\sup_{t\in\RR}\phi'(t)<c_2,~|\phi''(t)|\leq c_3\label{proof:2.1:Aeps:phicond1}\\
&|a'(t)|\leq \epsilon \phi'(t),~|\phi'''(t)|\leq \epsilon \phi'(t) \label{proof:2.1:Aeps:boundcond1}
\end{align}
and 
\begin{align}
&\inf_{t\in\RR}[a(t)+\alpha(t)]>c_1,~\sup_{t\in\RR}[a(t)+\alpha(t)]<c_2,\label{proof:2.1:Aeps:acond2}\\
&\inf_{t\in\RR}[\phi'(t)+\beta'(t)]>c_1,~\sup_{t\in\RR}[\phi'(t)+\beta'(t)]<c_2,~|\phi''(t)+\beta''(t)|\leq c_3\label{proof:2.1:Aeps:phicond2}\\
&|a'(t)+\alpha'(t)|\leq \epsilon (\phi'(t)+\beta'(t)),~|\phi'''(t)+\beta'''(t)|\leq \epsilon (\phi'(t)+\beta'(t)).  \label{proof:2.1:Aeps:boundcond2}
\end{align}

The proof is divided into two parts. The first part is determining the restrictions on the possible $\beta$ and $\alpha$ based on the positivity condition of $\phi'(t)$ and $a(t)$, which is independent of the conditions (\ref{proof:2.1:Aeps:boundcond1}) and (\ref{proof:2.1:Aeps:boundcond2}). The second part is to control the amplitude of $\beta$ and $\alpha$, which depends on the conditions (\ref{proof:2.1:Aeps:boundcond1}) and (\ref{proof:2.1:Aeps:boundcond2}).

First, based on the conditions (\ref{proof:2.1:Aeps:acond1}), (\ref{proof:2.1:Aeps:phicond1}), (\ref{proof:2.1:Aeps:acond2}) and (\ref{proof:2.1:Aeps:phicond2}), we show how $\beta$ and $\alpha$ are restricted. 
By the monotonicity of $\phi(t)$ based on the condition (\ref{proof:2.1:Aeps:phicond1}), define $t_m\in\RR$, $m\in\ZZ$, so that $\phi(t_m)=(m+1/2)\pi$ and $s_m\in\RR$, $m\in\ZZ$, so that $\phi(s_m)=m\pi$. In other words, we have 
\begin{align*}
g(t_m)=0~~\mbox{ and }~~g(s_m)=(-1)^ma(s_m).
\end{align*} 
Thus, for any $n\in\ZZ$, when $t=t_n$, we have
\begin{align}
&(a(t_n)+\alpha(t_n))\cos(\phi(t_n)+\beta(t_n))\nonumber\\
=&\,(a(t_n)+\alpha(t_n))\cos[n\pi+\pi/2+\beta(t_n)]\label{before_proof:identifiability:eq1}\\
=&\,a(t_n)\cos(n\pi+\pi/2)=0,\nonumber
\end{align}
where the second equality comes from (\ref{observation:identifiability:lemma:1}). This leads to $\beta(t_n)=k_n\pi$, $k_n\in\ZZ$, since $a(t_n)+\alpha(t_n)>0$ by (\ref{proof:2.1:Aeps:phicond2}). 
\begin{lemma}\label{proof:Claim1}
$k_n$ are the same for all $n\in\ZZ$ and $k_n$ are even. As changing the phase function globally by $2l\pi$, where $l\in\ZZ$, will not change the value of $g(t_n)$ for all $n\in \ZZ$, we could assume that $\beta(t_m)=0$ for all $m\in\ZZ$. 
\end{lemma} 
\begin{proof}
Suppose there exists $t_n$ so that $\beta(t_n)=k\pi$ and $\beta(t_{n+1})=(k+l)\pi$, where $k,l\in \ZZ$ and $l>0$. In other words, we have $\phi(t_{n+1})=\phi(t_n)+(l+1)\pi$. By the smoothness of $\beta$, we know there exists at least one $t'\in (t_n,t_{n_1})$ so that $\phi(t')+\beta(t')=(n+3/2)\pi$, but this is absurd since it means that $(a(t)+\alpha(t))\cos(\phi(t)+\beta(t))$ will change sign in $(t_n,t_{n+1})$ while $a(t)\cos(\phi(t))$ will not. 

Suppose $k_n$ is a fixed odd integer $k$, then since $\beta\in C^3(\RR)$ and $\beta(t_n)=\beta(t_{n+1})=k\pi$, there exists $t'\in(t_n,t_{n+1})$ so that $\beta(t')=k\pi$ and hence 
\begin{align*}
a(t')\cos(\phi(t'))&\,= (a(t')+\alpha(t'))\cos(\phi(t')+\beta(t'))=-(a(t')+\alpha(t'))\cos(\phi(t')),
\end{align*}
which is again absurd since $\cos(\phi(t'))\neq 0$ and the amplitudes are positive by (\ref{proof:2.1:Aeps:acond1}) and (\ref{proof:2.1:Aeps:acond2}). We thus obtain the second claim.
\end{proof}

\begin{lemma}\label{proof:Claim2}
$\beta'(t)$ is $0$ or changes sign inside $[t_n,\,t_{n+1}]$ for all $n\in\ZZ$. Furthermore, $|\beta(t')-\beta(t'')|<\pi$ for any $t',t''\in[t_m,t_{m+1}]$ for all $m\in\ZZ$.
\end{lemma}
\begin{proof}
By the fundamental theorem of calculus and the fact that $\beta(t_n)=\beta(t_{n+1})=0$, we know that
\begin{equation*}
0=\beta(t_{n+1})-\beta(t_n)=\int^{t_{n+1}}_{t_n}\beta'(u)\ud u.
\end{equation*}
which implies the first argument. Also, due to the monotonicity of $\phi+\beta$ (\ref{proof:2.1:Aeps:phicond2}), that is, $(n+1/2)\pi=\phi(t_{n})+\beta(t_{n})<\phi(t')+\beta(t')<\phi(t_{n+1})+\beta(t_{n+1})=(n+3/2)\pi$ for all $t'\in (t_{n},t_{n+1})$, we have the second claim
\begin{equation*}
|\beta(t')-\beta(t'')|<\pi.
\end{equation*} 
Indeed, if $|\beta(t')-\beta(t'')|\geq \pi$, for some $t',t''\in[t_n,t_{n+1}]$ and $t'<t''$, we get an contradiction since $\phi(t'')+\beta(t'')\notin[(n+1/2)\pi,(n+3/2)\pi]$ while $\phi(t')+\beta(t')\in[(n+1/2)\pi,(n+3/2)\pi]$.
\end{proof}

\begin{lemma}\label{proof:Claim3}
$\frac{a(s_n)}{a(s_n)+\alpha(s_n)}=\cos(\beta(s_n))$ for all $n\in\ZZ$. In particular, $\alpha(s_m)=0$ if and only if $\beta(s_m)=0$, $m\in\ZZ$.
\end{lemma}
\begin{proof}
When $t=s_m$, we have
\begin{align}
(-1)^ma(s_m)=&\,a(s_m)\cos(m\pi)\label{before_proof:identifiability:eq2}\\
=&\,(a(s_m)+\alpha(s_m))\cos[m\pi+\beta(s_m)]\nonumber\\
=&\,(-1)^m(a(s_m)+\alpha(s_m))\cos(\beta(s_m))\nonumber,
\end{align}
where the second equality comes from (\ref{observation:identifiability:lemma:1}), which leads to $\alpha(s_m)\geq0$ since $|\cos(\beta(s_m))|\leq 1$. 

Notice that (\ref{before_proof:identifiability:eq2}) implies that $\beta(s_m)=2k_m\pi$, where $k_m\in\ZZ$, if and only if $\alpha(s_m)=0$. Without loss of generality, assume $k_m>0$. Since $\beta\in C^3(\RR)$, there exists $t'\in(t_{m-1},s_m)$ so that $\beta(t')=\pi$ and hence 
\begin{align*}
a(t')\cos(\phi(t'))&\,= (a(t')+\alpha(t'))\cos(\phi(t')+\beta(t'))=-(a(t')+\alpha(t'))\cos(\phi(t')),
\end{align*}
which is absurd since $\cos(\phi(t'))\neq 0$ and the positive amplitudes by (\ref{proof:2.1:Aeps:acond1}) and (\ref{proof:2.1:Aeps:acond2}). Thus we conclude that $\beta(s_m)=0$.

To show the last part, note that when $\alpha(s_m)>0$, $0<\cos(\beta(s_m))=\frac{a(s_m)}{a(s_m)+\alpha(s_m)}<1$ by (\ref{before_proof:identifiability:eq2}). Thus, we know $\beta(s_m)\in(-\pi/2,\pi/2)+2n_m\pi$, where $n_m\in\ZZ$. By the same argument as in the above, if $n_m>0$, there exists $t'\in(t_{m-1},s_m)$ so that $\beta(t')=\pi$ and hence 
\begin{align*}
a(t')\cos(\phi(t'))&\,=\, (a(t')+\alpha(t'))\cos(\phi(t')+\beta(t'))=\,-(a(t')+\alpha(t'))\cos(\phi(t')),
\end{align*}
which is absurd since $\cos(\phi(t'))\neq 0$ and the positive amplitudes by (\ref{proof:2.1:Aeps:acond1}) and (\ref{proof:2.1:Aeps:acond2}). 
\end{proof}

\begin{lemma}\label{proof:Claim4}
$\frac{a(t_n)}{a(t_n)+\alpha(t_n)}=\frac{\phi'(t_n)+\beta'(t_n)}{\phi'(t_n)}$ for all $n\in\ZZ$. In particular, $\alpha(t_n)=0$ if and only if $\beta'(t_n)=0$, $n\in\ZZ$. 
\end{lemma}
\begin{proof}
For $0<x\ll 1$, we have
\begin{align}
(a(t_n+x)+\alpha(t_n+x))\cos(\phi(t_n+x)+\beta(t_n+x))=a(t_n+x)\cos(\phi(t_n+x)),\nonumber
\end{align}
which means that
\[
\frac{a(t_n+x)}{a(t_n+x)+\alpha(t_n+x)}=\frac{\cos(\phi(t_n+x)+\beta(t_n+x))}{\cos(\phi(t_n+x))}.
\]
By the smoothness of $\phi$ and $\beta$, as $x\to 0$, the right hand side becomes 
\begin{align*}
&\lim_{x\to0}\frac{\cos(\phi(t_n+x)+\beta(t_n+x))}{\cos(\phi(t_n+x))}\\
=&\,\lim_{x\to0}\frac{(\phi'(t_n+x)+\beta'(t_n+x)\sin(\phi(t_n+x)+\beta(t_n+x))}{\phi'(t_n+x)\sin(\phi(t_n+x))}\\
=&\,\frac{\phi'(t_n)+\beta'(t_n)}{\phi'(t_n)}.
\end{align*}
Thus, since $a(t_n+x)+\alpha(t_n+x)>0$ and $a(t_n+x)>0$ for all $x$, we have
\[
\frac{a(t_n)}{a(t_n)+\alpha(t_n)}=\frac{\phi'(t_n)+\beta'(t_n)}{\phi'(t_n)}.
\]
\end{proof}

\begin{lemma}\label{proof:Claim5}
$\beta''(t)$ is $0$ or changes sign inside $[t_n,\,t_{n+1}]$ for all $n\in\ZZ$.
\end{lemma}
\begin{proof}
This is clear since $\beta'(t)$ is 0 or changes sign inside $[t_n,\,t_{n+1}]$ for all $n\in\ZZ$ by Lemma \ref{proof:Claim2}.
\end{proof}

In summary, while $\beta(t_m)=0$ for all $m\in \ZZ$, in general we loss the control of $\alpha$ at $t_m$. On $s_m$, $\alpha$ is directly related to $\beta$ by Lemma \ref{proof:Claim3}; on $t_m$, $\alpha$ is directly related to $\beta'$ by Lemma \ref{proof:Claim4}. We could thus call $t_m$ and $s_m$ the {\em hinging points} associated with the function $g$. Note that the control of $\alpha$ and $\beta$ on the hinging points does not depend on $\beta''$'s condition.

To finish the second part of the proof, we have to consider the conditions (\ref{proof:2.1:Aeps:boundcond1}) and (\ref{proof:2.1:Aeps:boundcond2}).
\begin{lemma}
$|\alpha(t)|\leq 2\pi \epsilon$ for all $t\in\RR$. Further, we have $|\beta'(t_n)|\leq \frac{4\pi\phi'(t_n)}{a(t_n)}\epsilon$ for all $n\in\ZZ$.
\end{lemma}
\begin{proof}
Suppose there exists $t'$ so that $\alpha(t')> 2\pi \epsilon$. The case $\alpha(t')<-2\pi\epsilon$ can be proved in the same way. Take $m\in\ZZ$ so that $t'\in(t_m,t_{m+1}]$. 
From (\ref{proof:2.1:Aeps:boundcond1}) and (\ref{proof:2.1:Aeps:boundcond2}) we have 
\begin{align*}
|\alpha'(t)|\leq \epsilon(2\phi'(t)+\beta'(t)).
\end{align*}
Thus, take $t\in(t_m,t_{m+1})$. Without loss of generality, we could assume $t\in(t_m,t')$, we have by the fundamental theorem of calculus
\begin{align*}
|\alpha(t')-\alpha(t)|&~\leq \int_{t}^{t'}|\alpha'(u)|\ud u\leq \epsilon[2\phi(t')-2\phi(t)+\beta(t')-\beta(t)]\\
&~\leq \epsilon[(\phi(t_{m+1})+\beta(t_{m+1})-\phi(t_m)-\beta(t_m))+(\phi(t_{m+1})-\phi(t_m))]\leq 2\pi\epsilon,
\end{align*}
where the last inequality holds due to the fact that $\phi+\beta$ and $\phi$ are both monotonic and Lemma \ref{proof:Claim2}.
This fact leads to $\alpha(t)>0$ for all $t\in(t_m,t_{m+1}]$. Since $\beta(t_m)=0$ for all $m\in\ZZ$, there exists $\tilde{t}\in(t_m,t_{m+1})$ such that $\cos(\phi(\tilde{t})+\beta(\tilde{t}))>\cos(\phi(\tilde{t}))$. However, by the assumption and the above derivatives, we know that
\begin{align}
1>\frac{a(\tilde{t})}{a(\tilde{t})+\alpha(\tilde{t})}=\frac{\cos(\phi(\tilde{t})+\beta(\tilde{t}))}{\cos(\phi(\tilde{t}))},
\end{align}
which is absurd. Thus, we have obtained the first claim.

The second claim could be obtained by taking Lemma \ref{proof:Claim4} into account. Indeed, since $\beta'(t_n)=\frac{-\phi'(t_n)}{a(t_n)+\alpha(t_n)}\alpha(t_n)$ and $|\alpha(t)|\leq 2\pi\epsilon$, when $\epsilon$ is small enough, $|\beta'(t_n)|\leq \frac{2\phi'(t_n)}{a(t_n)}|\alpha(t_n)|\leq \frac{4\pi\phi'(t_n)}{a(t_n)}$.
\end{proof}
Thus we obtain the control of the amplitude. Note that the proof does not depend on the condition about $\beta''$.

\begin{lemma}\label{proof:Claim7}
$|\beta''(t)|\leq 2\pi \epsilon$, $|\beta'(t)|\leq \frac{2\pi\epsilon}{c_1} $ and $|\beta(t)|\leq \frac{2\pi\epsilon}{c_1^2} $ for all $t\in\RR$.
\end{lemma}
\begin{proof}
Suppose there existed $t'\in(t_m,t_{m+1})$  for some $m\in\ZZ$ so that $|\beta''(t')|>3\pi \epsilon$. Without loss of generality, we assume $\beta''(t')>0$.  
From (\ref{proof:2.1:Aeps:boundcond1}) and (\ref{proof:2.1:Aeps:boundcond2}) we have 
\begin{align*}
|\beta'''(t)|\leq \epsilon(2\phi'(t)+\beta'(t)).
\end{align*}
Thus, by the fundamental theorem of calculus, for any $t\in (t_m,t')$, we know 
\begin{align*}
|\beta''(t')-\beta''(t)|&\,\leq \int^{t'}_{t}|\beta'''(u)|\ud u\leq \epsilon \int^{t'}_{t}(2\phi'(t)+\beta'(t) )\ud u \leq 2\pi\epsilon,
\end{align*}
where the last inequality holds due to Lemma \ref{proof:Claim2} and the fact that $\phi(t')-\phi(t)\leq \phi(t_{m+1})-\phi(t_m)=\pi$ and $|\beta(t')-\beta(t)|<\pi$ from Lemma \ref{proof:Claim2}. Similarly, we have that for all $t\in (t',t_{m+1})$, $|\beta''(t')-\beta''(t)|\leq 2\pi\epsilon$.
Thus, $\beta''(t)>0$ for all $t\in[t_m,t_{m+1}]$, 
which contradicts the fact that $\beta''(t)$ must change sign inside $[t_m,t_{m+1}]$ by Lemma \ref{proof:Claim5}.

With the upper bound of $|\beta''|$, we immediately have for all $t\in[t_m,t_{m+1}]$ that
\[
|\beta'(t)-\beta'(t_m)|\leq \int^t_{t_m}|\beta''(u)|\ud u\leq 2\pi(t-t_m)\epsilon. 
\]
To bound the right hand side, note that $t-t_m\leq t_{m+1}-t_m\leq \frac{\pi}{\phi'(t')}$, where $t'\in [t_m,t_{m+1}]$. Since $|\beta''(t)|\leq 2\pi\epsilon$, when $\epsilon$ is small enough, $\frac{\pi}{\phi'(t')}\leq \frac{2\pi}{\phi'(t)}$. Thus, $|\beta'(t)-\beta'(t_m)|\leq2\pi(t-t_m)\epsilon\leq \frac{4\pi^2}{\phi'(t)}\epsilon$. To finish the proof, note that by Lemma \ref{proof:Claim4}, $|\beta'(t_m)|\leq \frac{4\pi\phi'(t_n)}{a(t_n)}\epsilon$. Similarly, we have the bound for $\beta$.
\end{proof}

\section{Proof of Theorem \ref{theorem:identifiability:multiple}}

When there are more than one gIMT in a given oscillatory signal $f\in \mathcal{Q}^{c_1,c_2,c_3}_{\epsilon,d}$, we loss the control of the hinging points for each gIMT like those, $t_m$ and $s_m$, in Theorem \ref{theorem:identifiability:single}. So the proof will be more qualitative.
Suppose $f=\tilde{f}\in \mathcal{Q}^{c_1,c_2,c_3}_{\epsilon,d}$, where
\[
f(t)=\sum_{l=1}^Na_l(t)\cos [2\pi\phi_l(t)],\quad \tilde{f}(t)=\sum_{l=1}^MA_l(t)\cos[2\pi\varphi_l(t)].
\]
Fix $t_0\in\RR$. Denote $f_{t_0}:=\sum_{l=1}^Nf_{t_0,l}$, $\tilde{f}_{t_0}:=\sum_{l=1}^M\tilde{f}_{t_0,l}$, 
\[
f_{t_0,l}(t):=a_l(t_0)\cos\left[2\pi\left(\phi_l(t_0)+\phi'_l(t_0)(t-t_0)+\phi''_l(t_0)\frac{(t-t_0)^2}{2}\right)\right]
\]
and
\[
\tilde{f}_{t_0,l}(t):=A_l(t_0) \cos\left[2\pi\left(\varphi_l(t_0) + \varphi'_l(t_0)(t-t_0)+\varphi''_l(t_0)\frac{(t-t_0)^2}{2}\right)\right].
\]
Note that $f_{t_0,l}$ is an approximation of $a_l(t)\cos [2\pi\phi_l(t)]$ near $t_0$ based on the assumption of $\mathcal{Q}^{c_1,c_2,c_3}_{\epsilon,d}$, where we approximate the amplitude $a_l(t)$ by the zero-th order Taylor expansion and the phase function $\phi_l(t)$ by the second order Taylor expansion. To simplify the proof, we focus on the case that $|\phi''_l(t_0)|>\epsilon|\phi'_l(t_0)|$ and $|\varphi''_l(t_0)|>\epsilon|\varphi'_l(t_0)|$ for all $l$. For the case when there is one or more $l$ so that $|\phi''_l(t_0)|\leq\epsilon |\phi'_l(t_0)|$, the proof follows the same line while we approximate the phases of these oscillatory components by the first order Taylor expansion. 

Recall that the short time Fourier transform (STFT) of a given tempered distribution $\mathsf{f}\in \mathcal{S}'$ associated with a Schwartz function $g\in \mathcal{S}$ as the window function is defined as
\[
V^{(\mathsf{g})}_{\mathsf{f}}(t,\eta):=\int_{\RR}\mathsf{f}(x)g(x-t)e^{-i2\pi\eta x}\ud x.
\]
Note that by definition $f,f_{t_0},\tilde{f}_{t_0}\in \mathcal{S}'$. To prove the theorem, we need the following lemma about the STFT.

\begin{lemma}\label{lemma:identifiability:Wf_expansion}
For a fixed $t_0\in\RR$, we have
\begin{align}
\Big|V^{(g)}_{f}(\tau,\eta)&-V^{(g)}_{f_{t_0}}(\tau,\eta)\Big|=O(\epsilon).\nonumber
\end{align}
where $C$ is a universal constant depending on $c_1$, $c_2$ and $d$. 
\end{lemma}

\begin{proof}
Fix a time $t_0\in\RR$. By the same argument as that in \cite{Daubechies_Lu_Wu:2011,Chen_Cheng_Wu:2014} and the conditions of $\mathcal{Q}_{\epsilon,d}^{c_1,c_2,c_3}$, we immediately have
\begin{align*}
&|V^{(g)}_{f}(\tau,\eta)-V^{(g)}_{f_{t_0}}(\tau,\eta)|\\
=\,&\left|\int_{\RR}(f(t)-f_{t_0}(t))g(t-\tau)e^{-i2\pi\eta t}\ud t\right|\\
\leq\,&\sum_{l=1}^N\left|\int_{\RR}(a_l(t)-a_l(t_0))\cos [2\pi\phi_l(t)]g(t-t_0)e^{-i2\pi\eta t}\ud t\right|\\
&+\sum_{l=1}^N\left|\int_{\RR}a_l(t_0)\big(\cos [2\pi\phi_l(t)]-\cos[2\pi(\phi(t_0)+\phi'(t_0)(t-t_0)\right.\\
&\qquad\qquad\left.+\frac{1}{2}\phi''(t_0)(t-t_0)^2)]\big)g(t-t_0)e^{-i2\pi\eta t}\ud t\right|\\
=\,&O(\epsilon),   
\end{align*}
where the last term depends only on the first few absolute moments of $g$ and $g'$, $d$, $c_1$ and $c_2$.
\end{proof}

With this claim, we know in particular that $V^{(g)}_f(t_0,\eta)=V^{(g)}_{f_{t_0}}(t_0,\eta)+O(\epsilon)$. As a result, the spectrogram of $f$ and $f_{t_0}$ are related by
\[
|V^{(g)}_f(t_0,\eta)|^2=|V^{(g)}_{f_{t_0}}(t_0,\eta)|^2+O(\epsilon).
\]
Next, recall that the spectrogram of a signal is intimately related to the Wigner-Ville distribution in the following way
\[
|V^{(g)}_{f_{t_0}}(\tau,\eta)|^2=\int\int WV_{f_{t_0}}(x,\xi)WV_{g}(x-\tau,\xi-\eta)\ud x\ud\xi,
\]
where the Wigner-Ville distribution of a function $h$ in the suitable space is defined as
\[
WV_{h}(x,\xi):=\int h(x+\tau/2)h^*(x-\tau/2)e^{-i2\pi \tau \xi}\ud \tau.
\]

\begin{lemma}\label{proof:Theorem2:claim2}
Take $g(t)=(2\sigma)^{1/4}\exp\left\{-\pi\sigma t^2\right\}$, where $\sigma=c_3$. When $d$ is large enough described in (\ref{proof:Thm2:dlowerbound}), we have
\[
|V^{(g)}_{f}(t_0,\eta)|^2=L(t_0,\eta)+\epsilon \quad\mbox{and}\quad |V^{(g)}_{\tilde{f}}(t_0,\eta)|^2=\widetilde{L}(t_0,\eta)+\epsilon,
\]
where
$$
L(t_0,\eta):=\sum_{l=1}^{N}a^2_l(t_0)\sqrt{\frac{\sigma}{2(\sigma^2+\phi''_l(t_0)^2)}}\exp\left\{-\frac{2\pi\sigma(\phi_l'(t_0)-\eta)^2}{\sigma^2+\phi''_l(t_0)^2}\right\}
$$ 
and
$$
\widetilde{L}(t_0,\eta):=\sum_{l=1}^{M}A^2_l(t_0)\sqrt{\frac{\sigma}{2(\sigma^2+\varphi''_l(t_0)^2)}}\exp\left\{-\frac{2\pi\sigma(\varphi_l'(t_0)-\eta)^2}{\sigma^2+\varphi''_l(t_0)^2}\right\}.
$$
\end{lemma}

\begin{proof}
By a direct calculation, the Wigner-Ville distribution of the Gaussian function $g(t)=(2\sigma)^{1/4}\exp\left\{-\pi\sigma t^2\right\}$ with the unit energy, where $\sigma>0$, is 
\[
WV_g(x,\xi)=2\exp\left\{-2\pi\left(\sigma x^2+\frac{\xi^2}{\sigma}\right)\right\};
\]
similarly, the Wigner-Ville distribution of $f_{t_0,l}$ is
\[
WV_{f_{t_0,l}}(x,\xi)=a^2_l(t_0)\delta_{\phi'_l(t_0)+\phi''_l(t_0)(x-t_0)}(\xi).
\]
Thus, we know 
\begin{align}
&\left|V^{(g)}_{f_{t_0,l}}(t_0,\eta)\right|^2\label{proof:Thm2:Vonechirp}\\
=&\int\int WV_{f_{t_0,l}}(x,\xi)WV_{g}(x-t_0,\xi-\eta)\ud x\ud\xi\nonumber\\
=&\int\int \left(a^2_l(t_0)\delta_{\phi'_l(t_0)+\phi''_l(t_0)(x-t_0)}(\xi)\right)2\exp\left\{-2\pi\left(\sigma (x-t_0)^2+\frac{(\xi-\eta)^2}{\sigma}\right)\right\}\ud \xi\ud x\nonumber\\
=&2a^2_l(t_0)\int \exp\left\{-2\pi\left(\sigma (x-t_0)^2+\frac{(\phi_l'(t_0)+\phi_l''(t_0)(x-t_0)-\eta)^2}{\sigma}\right)\right\}\ud x\nonumber\\
=&a^2_l(t_0)\sqrt{\frac{\sigma}{2(\sigma^2+\phi''_l(t_0)^2)}}\exp\left\{-\frac{2\pi\sigma}{\sigma^2+\phi''_l(t_0)^2}(\phi_l'(t_0)-\eta)^2\right\}.\nonumber
\end{align}
Thus, we have the expansion of $\sum_{l=1}^N|V^{(g)}_{f_{t_0,l}}(t_0,\eta)|^2$, which is $L(t_0,\eta)$.
Next, we clearly have 
\begin{align*}
&\left|V^{(g)}_{f_{t_0}}(\tau,\eta)|^2-\sum_{l=1}^N|V^{(g)}_{f_{t_0,l}}(\tau,\eta)|^2\right|=\left|\Re \sum_{k\neq l}V^{(g)}_{f_{t_0,l}}(\tau,\eta) \overline{V^{(g)}_{f_{t_0,k}}(\tau,\eta)}\right|\\
\leq& \sum_{k\neq l}\left|V^{(g)}_{f_{t_0,l}}(\tau,\eta)\right| \left|V^{(g)}_{f_{t_0,k}}(\tau,\eta)\right|.
\end{align*}
To bound the right hand side, note that (\ref{proof:Thm2:Vonechirp}) implies 
$$
\left|V^{(g)}_{f_{t_0,l}}(t_0,\eta)\right|= a_l(t_0)\left(\frac{\sigma}{2(\sigma^2+\phi''_l(t_0)^2)}\right)^{1/4}\exp\left\{-\frac{\pi\sigma(\phi_l'(t_0)-\eta)^2}{\sigma^2+\phi''_l(t_0)^2}\right\}.
$$
As a result, $\sum_{k\neq l}\left|V^{(g)}_{f_{t_0,l}}(t_0,\eta)\right| \left|V^{(g)}_{f_{t_0,k}}(t_0,\eta)\right|$ becomes
\begin{align*}
&\, \sum_{k\neq l}\frac{a_k(t_0)a_l(t_0)\sigma^{1/2}}{\left(4(\sigma^2+\phi''_k(t_0)^2)(\sigma^2+\phi''_l(t_0)^2)\right)^{1/4}}\exp\left\{-\pi\sigma\left(\frac{(\phi_k'(t_0)-\eta)^2}{\sigma^2+\phi''_k(t_0)^2}+\frac{(\phi_l'(t_0)-\eta)^2}{\sigma^2+\phi''_l(t_0)^2}\right)\right\},
\end{align*}
which is a smooth and bounded function of $\eta$ and is bounded by
\begin{align}\label{proof:Thm2:bound2}
\frac{c_2^2}{\sqrt{2}\sigma^{1/2}}\sum_{k\neq l}\exp\left\{-\frac{\pi\sigma}{\sigma^2+c_3^2}\left[(\phi_k'(t_0)-\eta)^2+(\phi_l'(t_0)-\eta)^2\right]\right\}.
\end{align}
Suppose the maximum of the right hand side is achieved when $\eta=\phi_{k_0}'(t_0)$, where $k_0=\arg\min_{l=1,\ldots,N-1}(\phi_{l+1}'(t_0)-\phi_{l}'(t_0))$. To bound (\ref{proof:Thm2:bound2}), denote $\gamma=\frac{\pi\sigma}{\sigma^2+c_3^2}$ to simplify the notation. Clearly, the summation in (\ref{proof:Thm2:bound2}) is thus bounded by
\begin{align*}
&2\sum_{l=1}^\infty e^{-l^2\gamma d^2}+2\sum_{k=1}^\infty e^{-k^2\gamma d^2}\sum_{l=0}^\infty e^{-l^2\gamma d^2}\leq  2(Q+1)S,
\end{align*}
where $Q=\int_0^\infty e^{-l^2\gamma d^2t}\ud t=\frac{\sqrt{\pi}}{2\gamma d^2}$ and $S=\int_1^\infty e^{-l^2\gamma d^2t}\ud t\leq \frac{d\sqrt{\gamma}}{e^{-\gamma d^2}}$ . Note that here we take the bounds $\sum_{l=0}^\infty e^{-l^2\gamma d^2}\leq Q$ and $\sum_{l=1}^\infty e^{-l^2\gamma d^2}\leq S$. Thus, we conclude that the interference term is bounded by $\frac{\sqrt{2}c_2^2}{\sigma^{1/2}}(Q+1)S$. To finish the proof, we require the interference term to be bounded by $\epsilon$, which leads to the following bound of $d$ when we take $\sigma=c_3$:
\begin{align}\label{proof:Thm2:dlowerbound}
d\geq \sqrt{2\ln c_2+\frac{1}{2}\ln c_3-\ln \epsilon}.
\end{align}
We have finished the proof.
\end{proof}

Since $t_0$ is arbitrary in the above argument and the spectrogram of a function is unique, we have $|V^{(g)}_{f}(t_0,\eta)|^2=|V^{(g)}_{\tilde{f}}(t_0,\eta)|^2$ and hence
\begin{equation}\label{proof:identifiability:multiple:1}
|L(t_0,\eta)-\widetilde{L}(t_0,\eta)|=O(\epsilon).
\end{equation}
With the above claim, we now show $M=N$. 
\begin{lemma}\label{proof:Theorem2:claim3}
$M=N$.
\end{lemma}
\begin{proof}
With $\sigma=c_3$, by Lemma \ref{proof:Theorem2:claim2}, for each $l=1,\ldots, N$, there exists a subinterval $I_l(t_0)$ around $\phi'_l(t_0)$ so that on $I_l(t_0)$, $L(t_0,\eta)>\frac{a_l^2(t_0)c_3^2}{2\sqrt{2(c_3^2+\phi''_l(t_0)^2)}}>\frac{c^2_1c_3^2}{2\sqrt{2c_3^2+2c_2^2}}$. 
Similarly, for each $l=1,\ldots,M$, there exists a subinterval $J_l(t_0)$ around $\varphi'_l(t_0)$ so that on $J_l(t_0)$, $\widetilde{L}(t_0,\eta)>\frac{A_l^2(t_0)c_3^2}{2\sqrt{2(c_3^2+\varphi''_l(t_0)^2)}}>\frac{c^2_1c_3^2}{2\sqrt{2c_3^2+2c_2^2}}$. 
Thus, when $\epsilon$ is small enough, in particular, $\epsilon\ll\frac{c^2_1c_3^2}{2\sqrt{2c_3^2+2c_2^2}}$, the equality in (\ref{proof:identifiability:multiple:1}) cannot hold if $M\neq N$. 
\end{proof} 
With this claim, we obtain the first part of the proof, and hence the equality
\begin{align}\label{proof:Theorem2:f:expansion}
f(t)=\sum_{l=1}^Na_l(t)\cos [2\pi\phi_l(t)]=\sum_{l=1}^NA_l(t)\cos[2\pi\varphi_l(t)]\in \mathcal{Q}^{c_1,c_2,c_3}_{\epsilon,d}.
\end{align}
Now we proceed to finish the proof. Note that it is also clear that the sets $I_l(t_0)$ and $J_l(t_0)$ defined in the proof of Lemma \ref{proof:Theorem2:claim3} satisfy that $I_l(t_0)\cap I_k(t_0)= \emptyset$ for all $l\neq k$. Also, $I_l(t_0)\cap J_l(t_0)\neq \emptyset$ and $I_l(t_0)\cap J_k(t_0)=\emptyset$ for all $l\neq k$. Indeed, if $k=l+1$ and we have $I_l(t_0)\cap J_{l+1}(t_0)\neq\emptyset$, then $L(t_0,\eta)> \frac{c^2_1c_3^2}{2\sqrt{2c_3^2+2c_2^2}}$ on $J_{l+1}(t_0)\backslash I_{l}(t_0)$, which leads to the contradiction. By the ordering of $\phi'_l(t_0)$ and hence the ordering of $I_l(t_0)$, we have the result.

Take $\ell=1$ and $\eta=\phi'_1(t_0)$. By Lemma \ref{proof:Theorem2:claim2}, when $d$ is large enough, on $I_1(t_0)$ we have 
\begin{align}\label{proof:Thm2:Aadifferecne}
\frac{a^2_1(t_0)}{\sqrt{2(1+\phi''_1(t_0)^2)}}= \frac{A^2_1(t_0)}{\sqrt{2(1+\varphi''_1(t_0)^2)}}\exp\left\{-\frac{2\pi c_3(\varphi_1'(t_0)-\phi_1'(t_0))^2}{c_3^2+\varphi''_1(t_0)^2}\right\}+O(\epsilon),
\end{align}
which leads to the fact that
\begin{align}\label{proof:Thm2:Aadifferecne2}
\left|\frac{a^2_1(t_0)c_3^2}{\sqrt{2(c_3^2+\phi''_1(t_0)^2)}}-\frac{A^2_1(t_0)c_3^2}{\sqrt{2(c_3^2+\varphi''_1(t_0)^2)}}\right|=O(\epsilon).
\end{align}
Indeed, without loss of generality, assume $\frac{a^2_1(t_0)c_3^2}{\sqrt{2(c_3^2+\phi''_1(t_0)^2)}}\geq\frac{A^2_1(t_0)c_3^2}{\sqrt{2(c_3^2+\varphi''_1(t_0)^2)}}$ and we have 
\begin{align*}
&\frac{a^2_1(t_0)c_3^2}{\sqrt{2(c_3^2+\phi''_1(t_0)^2)}}-\frac{A^2_1(t_0)c_3^2}{\sqrt{2(c_3^2+\varphi''_1(t_0)^2)}}\\
\leq\,&\frac{a^2_1(t_0)c_3^2}{\sqrt{2(c_3^2+\phi''_1(t_0)^2)}}-\frac{A^2_1(t_0)c_3^2}{\sqrt{2(c_3^2+\varphi''_1(t_0)^2)}}\exp\left\{-\frac{2\pi c_3(\varphi_1'(t_0)-\phi_1'(t_0))^2}{c_3^2+\varphi''_1(t_0)^2}\right\} =O(\epsilon)
\end{align*} 
by (\ref{proof:Thm2:Aadifferecne}) since $0$ is the unique maximal point of the chosen Gaussian function. 

\begin{lemma}\label{proof:Theorem2:claim4}
$|\phi'_\ell(t)-\varphi'_\ell(t)|=O(\sqrt{\epsilon})$ for all time $t\in\RR$ and $\ell=1,\ldots,N$.
\end{lemma}
\begin{proof}
Fix $t_0\in\RR$ and $\ell=1$. By (\ref{proof:Thm2:Aadifferecne}), (\ref{proof:Thm2:Aadifferecne2}) and the conditions of $\mathcal{Q}^{c_1,c_2,c_3}_{\epsilon,d}$, on $I_1(t_0)$ we have
\[
\frac{A^2_1(t_0)c_3^2}{\sqrt{2(c_3^2+\varphi''_1(t_0)^2)}}\left|1-\exp\left\{-\frac{2\pi c_3(\varphi_1'(t_0)-\phi_1'(t_0))^2}{c_3^2+\varphi''_1(t_0)^2}\right\}\right|=O(\epsilon).
\]
Due to the fact that the Gaussian function monotonically decreases as $\frac{2\pi c_3(\varphi_1'(t_0)-\phi_1'(t_0))^2}{c_3^2+\varphi''_1(t_0)^2}>0$, we have 
\[
\frac{(\varphi_1'(t_0)-\phi_1'(t_0))^2}{c_3^2+\varphi''_1(t_0)^2}=O(\epsilon).
\]
Since $\varphi''_1$ is uniformly bounded by $c_2$, we know
\[
|\varphi_1'(t_0)-\phi_1'(t_0)|=O(\sqrt{\epsilon}). 
\]
By the same argument, we know that $|\varphi_l'(t)-\phi_l'(t)|=O(\sqrt{\epsilon})$ for all $l=1,\ldots,N$ and $t\in\RR$. 
\end{proof}

\begin{lemma}\label{proof:Theorem2:claim5}
$|\phi''_\ell(t)-\varphi''_\ell(t)|=O(\sqrt{\epsilon})$ for all time $t\in\RR$ and $\ell=1,\ldots,N$.
\end{lemma}
\begin{proof}
Fix $t_0\in\RR$ and $\ell=1$. By the assumption that $\phi'''_1(t_0)=O(\epsilon)$ and $\varphi'''_1(t_0)=O(\epsilon)$, we claim that $|\phi''_1(t_0)-\varphi''_1(t_0)|=O(\sqrt{\epsilon})$ holds. 
Indeed, we have
\[
\phi_1'({t_0+1})=\phi_1'(t_0)+\int_{t_0}^{{t_0+1}}\phi''_1(s)\ud s\quad\mbox{and}\quad\varphi_1'({t_0+1})=\varphi_1'(t_0)+\int_{t_0}^{{t_0+1}}\varphi''_1(s)\ud s,
\]
which leads to the relationship
\[
\phi_1'({t_0+1})-\varphi_1'({t_0+1})=\phi_1'(t_0)-\varphi_1'(t_0)+\int_{t_0}^{{t_0+1}}(\phi''_1(s)-\varphi''_1(s))\ud s.
\]
Therefore, by the assumption that $\phi'''_1(t_0)=O(\epsilon)$ and $\varphi'''_1(t_0)=O(\epsilon)$, we have
\begin{align*}
&\int_{t_0}^{{t_0+1}}(\phi''_1(s)-\varphi''_1(s))\ud s\\
=&\,\int_{t_0}^{{t_0+1}}\left(\phi''_1(t_0)-\varphi''_1(t_0)+\int_{t_0}^s\left(\phi'''_1(x)-\varphi'''_1(x)\right)\ud x\right)\ud s\\
=&\,\phi''_1(t_0)-\varphi''_1(t_0)+O(\epsilon),
\end{align*}
which means that $|\phi''_1(t_0)-\varphi''_1(t_0)|=O(\sqrt{\epsilon})$ since $|\phi_1'({t_0+1})-\varphi_1'({t_0+1})|=O(\sqrt{\epsilon})$ and $|\phi_1'({t_0})-\varphi_1'({t_0})|=O(\sqrt{\epsilon})$. By the same argument, we know that $|\varphi_l''(t)-\phi_l''(t)|=O(\sqrt{\epsilon})$ for all $l=1,\ldots,N$ and $t\in\RR$.
\end{proof}

\begin{lemma}\label{proof:Theorem2:claim6}
$|a_\ell(t)-A_\ell(t)|=O(\sqrt{\epsilon})$ for all time $t\in\RR$ and $\ell=1,\ldots,N$.
\end{lemma}
\begin{proof}
Fix $t_0\in\RR$ and $\ell=1$. From (\ref{proof:Thm2:Aadifferecne2}), it is clear that $|a_1(t_0)-A_1(t_0)|=O(\sqrt{\epsilon})$ if and only if $|\phi''_1(t_0)-\varphi''_1(t_0)|=O(\sqrt{\epsilon})$, so we obtain the claim by Lemma \ref{proof:Theorem2:claim5}.
Similar argument holds for all time $t\in\RR$ and $\ell=2,\ldots,N$. 
\end{proof}
Lastly, we show the difference of the phase functions. 
\begin{lemma}
$|\phi_\ell(t)-\varphi_\ell(t)|=O(\sqrt{\epsilon})$ for all time $t\in\RR$ and $\ell=1,\ldots,N$.
\end{lemma}
\begin{proof}
By (\ref{proof:Theorem2:f:expansion}) and the fact that $|a_l(t)-A_l(t)|=O(\sqrt{\epsilon})$, we have for all $t\in\RR$, 
\[
\sum_{l=1}^Na_l(t)\cos [2\pi\phi_l(t)]=\sum_{l=1}^Na_l(t)\cos[2\pi(\phi_l(t)+\alpha_l(t))]+O(\sqrt{\epsilon}),
\]
where $\alpha_l\in C^3(\RR)$. Note that $\sum_{l=1}^Na_l(t)\cos[2\pi(\phi_l(t)+\alpha_l(t))]\in \mathcal{Q}^{c_1,c_2,c_3}_{\epsilon,d}$. Fix $t_0\in\RR$. 
Suppose there exists $t_0$ and the smallest number $k$ so that $\alpha_k(t_0)=O(\sqrt{\epsilon})$ up to multiples of $2\pi$ does not hold. Then there exists at least one $\ell\neq k$ so that $\alpha_\ell(t_0)=O(\sqrt{\epsilon})$ does not hold. Suppose $L>k$ is the largest integer that $\alpha_L(t_0)=O(\sqrt{\epsilon})$ does not hold. In this case, there exists $t_1>t_0$ so that $\sum_{l=1}^Na_l(t_1)\cos [2\pi\phi_l(t_1)]=\sum_{l=1}^Na_l(t_1)\cos[2\pi(\phi_l(t_1)+\alpha_l(t_1))]+O(\sqrt{\epsilon})$ does not hold. Indeed, as $\phi_L'(t_0)$ is higher than $\phi_k'(t_0)$ by at least $d$, we could find $t_1=\phi^{-1}_k(\phi_k(t_0)+c)$, where $0<c<\pi$, so that $\cos[2\pi(\phi_L(t_1)+\alpha_L(t_1))]-\cos[2\pi(\phi_L(t_1))]=\cos[2\pi(\phi_L(t_0)+\alpha_L(t_0))]-\cos[2\pi(\phi_L(t_0))]+O(\sqrt{\epsilon})$ does not hold while $\sum_{l\neq L}^Na_l(t)\cos [2\pi\phi_l(t)]=\sum_{l\neq L}^Na_l(t)\cos[2\pi(\phi_l(t)+\alpha_l(t))]+O(\sqrt{\epsilon})$ holds. We thus get a contradiction and hence the proof.
\end{proof}

\section{A convergence study of Algorithm~\ref{alg:alter}}
\label{appendix:convergence}

We provide here a simple convergence study of Algorithm~\ref{alg:alter}, based on the Zangwill's global convergence theorem~\cite{zangwill1969nonlinear} which can be stated as follow

\begin{theorem}
	Let $\mathcal{A}$ be an algorithm on $\mathcal{X}$, and suppose that, given $x_0 \in \mathcal{X}$, the
	sequence $\left\{x_k\right\}_{k=1}^{\infty}$ is generated and satisfies
	\[
		x_{k+1} \in \mathcal{A}\left(x_k\right)\ .
	\]
	Let a solution set $\Gamma$ be given, and suppose that
	\begin{enumerate}
		\item[(i)] the sequence $\left\{x_k\right\}_{k=0}^{\infty}\subset \mathcal{S}$ for $\mathcal{S}\subset \mathcal{X}$ a compact set.
		\item[(ii)] there is a continuous function $\mathcal{Z}$ on $\mathcal{X}$ such that 
		\begin{enumerate}
			\item if $x\notin\Gamma$, then $\mathcal{Z}(y)< \mathcal{Z}(x)$ for all $y\in\mathcal{A}(x)$
			\item if $x\in\Gamma$, then $\mathcal{Z}(y)\leq \mathcal{Z}(x)$ for all $y\in\mathcal{A}(x)$
		\end{enumerate}
		\item[(iii)] the mapping $\mathcal{A}$ is closed at all point $\mathcal{X}\ \Gamma$
	\end{enumerate}
	Then the limit of any convergent subsequence of $\left\{x_k\right\}_{k=0}^{\infty}$ is a solution, and $\mathcal{Z}_{k}\rightarrow\mathcal{Z}(x^*)$ for some $x^*\in\Gamma$.
\end{theorem}

The algorithm $\mathcal{A}$ is here the alternating minimization Alg.~\ref{alg:alter}. The solution set $\Gamma$ is then naturally the set of 
the critical points of the functional $\mathcal{H}$. Equivalenly, $\Gamma$ is the set of the fixed point of Alg.~\ref{alg:alter}. Indeed, for
$\valpha$ being fixed, $\vF^*$ is a minimizer of $\mathcal{H}_{\valpha}$ is equivalent to $\vF^*$ is a fixed point of Alg.~\ref{alg:FISTA}~\cite{combettes2005signal}. The descent function $\mathcal{Z}$ is then naturraly the functional $\mathcal{H}$.

As we work in the finite dimensional case, the boundeness of the sequence and then point (i) of the Theorem is here direct consequences of point (iii). Moreover, thanks to the continuity of the soft-thresholding operator, the mapping $\mathcal{A}$ is continuous and point (iii) is also a direct consequence of point (iii).

Point (iii) of the theorem comes from the monotonic version of FISTA. Indeed, the very first iteration of FISTA is equivalent to a "simple" forward-backward step, which ensure the strict decreasing of the functional $\mathcal{H}_{\valpha}$ (see~\cite{tseng2010approximation}). Then, as $\valpha$ is the unique minimizer of the function $\mathcal{H}_{\vF}$, we have a sequence 
 $\left\{(\valpha_k,\vF_k)\right\}_{k=0}^{\infty}$ 
 such that	$\mathcal{H}\left(\valpha^{k+1},\vF^{k+1}\right) < \mathcal{H}\left(\valpha^{k},\vF^{k}\right)$ as soon as  $\left(\valpha^k,\vF^k\right)$ is not a critical point of $\mathcal{H}$.

\end{document}